\documentclass[twoside,a4paper]{amsart}
\usepackage{amsmath,amssymb,amsthm}
\usepackage{graphicx}
\usepackage[english]{babel}
\usepackage{hyperref}
\usepackage{color}
\usepackage{enumerate}
\usepackage[margin=1.5in,head=13.6pt]{geometry}

\theoremstyle{plain}
\newtheorem{thm}{Theorem}[section]

\newtheorem{lem}[thm]{Lemma}
\newtheorem{prop}[thm]{Proposition}
\newtheorem{res}[thm]{Result}
\theoremstyle{definition}
\newtheorem{defn}[thm]{Definition}

\usepackage{color}

\DeclareMathOperator{\ome}{\Omega}
\DeclareMathOperator{\Aut}{Aut}
\DeclareMathOperator{\Tr}{Tr}
\DeclareMathOperator{\supp}{supp}

\newcommand{\tmat}[4]{\left\lfloor\begin{smallmatrix}#1&#2\\#3&#4\end{smallmatrix}\right\rceil}
\begin{document}

\title[Strongly regular graphs with 2-transitive two-graphs]{Vertex-transitive strongly regular graphs in the switching class of doubly transitive two-graphs}

\author{Robert F. Bailey}
\address{School of Science and the Environment (Mathematics), Memorial University--Grenfell Campus, 20 University Drive, Corner Brook, NL A2H 6P9, Canada.}
\email{robert.bailey@mun.ca}
\author{G\'{a}bor P. Nagy}
\address{Bolyai Institute, University of Szeged, Aradi V\'{e}rtan\'{u}k tere 1, H-6720 Szeged, Hungary.}
\email{nagyg@math.u-szeged.hu}
\author{Valentino Smaldore}
\address{Dipartimento di Tecnica e Gestione dei Sistemi Industriali,
Universit\`{a} degli Studi di Padova, Stradella S. Nicola 3, 36100 Vicenza, Italy.}
\email{valentino.smaldore@unipd.it}


\begin{abstract}
Let $\mathcal{T}$ be a finite two-graph whose full automorphism group $G$ is $2$-transitive. Up to complementation, we classify the primitive strongly regular graphs with associated two-graph $\mathcal{T}$ whose full automorphism group is a transitive maximal subgroup of $G$ and does not contain the largest perfect subgroup of $G$. In doing so, we obtain a new family of vertex-transitive strongly regular graphs whose associated two-graph arises from $P\Sigma L(2,q)$.
\end{abstract}
\maketitle
\tableofcontents

\section{Introduction}
A \textit{strongly regular graph} with parameters $(v,k,\lambda,\mu)$ is a $k$-regular graph with $v$ vertices, where any two adjacent vertices have $\lambda$ common neighbours, and any two nonadjacent vertices have $\mu$ common neighbours. A strongly regular graph is said to be \textit{primitive} if both it and its complement are connected, and is {\em imprimitive} otherwise; it is well-known that the only imprimitive strongly regular graphs are the complete multipartite graphs $K_{t,t,\ldots,t}$ or their complements.  We refer to the monograph of Brouwer and Van Maldeghem \cite{Brouwer2022} for more details. 

A \textit{two-graph} is a pair $(V, T)$, where $T$ is a set of unordered triples of a vertex set $V$, such that every (unordered) quadruple from $V$ contains an even number of triples in $T$; see \cite[Chapter 10]{Brouwer2012} for more information. Given a graph $\Gamma = (V, E)$, the set of triples $T$ of the vertex set $V$ whose induced subgraph has an odd number of edges is the \textit{associated two-graph} of $\Gamma$.  A two-graph is \textit{trivial} if either $T$ is empty or contains all unordered triples of $V$.  Two-graphs were introduced by G.\ Higman and studied in more detail by Taylor \cite{Taylor1971,Taylor1977}.

Two simple graphs are said to be \textit{switching equivalent} if we can obtain one from the other by \textit{Seidel switching}, which exchanges edges and nonedges across a partition of the vertex set \cite{Seidel1974}. This operation preserves the spectrum of the Seidel matrix. Two simple graphs on the same vertex set are switching equivalent if and only if they have the same associated two-graph.

Finite nontrivial two-graphs with $2$-transitive automorphism groups $G$ have been classified by Taylor \cite{Taylor1992}:
\begin{enumerate}[1)]
\item \textbf{Affine polar type:} $G\cong \mathbb{F}_2^{2m}\rtimes Sp(2m,2)$, $m\geq 2$;
\item \textbf{Symplectic type:} $G\cong Sp(2m,2)$, $m\geq 3$;
\item \textbf{Linear type:} $G\cong P\Sigma L(2,q)$, for $q\equiv 1\pmod{4}$;
\item \textbf{Unitary type:} $G \cong P\Gamma U(3,q)$, for $q\geq 5$ odd;
\item \textbf{Ree type:} $G\cong Ree(q)\rtimes \mathrm{Aut}(\mathbb F_q)$, for $q=3^{2e+1}$, $e\geq 1$;
\item \textbf{Sporadic type:} $G\cong HS$ or $G\cong Co_3$.
\end{enumerate}
In particular, $G$ is either almost simple or a primitive group of affine type.  We note that some restrictions are imposed to ensure that there are no overlaps between the six classes.  We exclude $Sp(4,2)$ since $Sp(4,2)\cong P\Sigma L(2,9)\cong S_6$, which acts 2-transitively on 10 points; this action yields the two-graph of the Petersen graph (in our language, this is a two-graph of linear type). 
We also exclude the unitary group $P\Gamma U(3,3)$ and the smallest Ree group $Ree(3)$: both are maximal subgroups of $Sp(6,2)$, and the corresponding two-graphs are, in fact, a symplectic type two-graph on 28 points.

Let $G^{(\infty)}$ denote the unique largest perfect subgroup of $G$.  Equivalently, $G^{(\infty)}$ is the smallest normal subgroup of $G$ with a solvable factor group. 
The main result of this paper is the following.
\begin{thm} \label{thm:main}
Let $\Gamma=(V,E)$ be a primitive strongly regular graph with associated two-graph $\mathcal{T}=(V,T)$. Write $H=\Aut(\Gamma)$ and $G=\Aut(\mathcal{T})$. Assume that $G$ is $2$-transitive and that $H$ is a transitive maximal subgroup of $G$ that does not contain $G^{(\infty)}$. Then, up to isomorphism and graph complementation, $G$, $H$, and $\Gamma$ are as in Table~\ref{tab:main}.
\end{thm}

\begin{table}
\centering
\caption{Vertex-transitive strongly regular graphs in $2$-transitive two-graphs, up to graph complementation}
\label{tab:main}
\small
\begin{tabular}{@{}p{17mm}p{24mm}p{26mm}p{31mm}p{19mm}@{}}
$G$ & $H$ & $\Gamma$ & Properties & References  \\ \hline\hline
$ASp(2m,2)$ & $\mathbb{F}_2^{2m}\rtimes O^\pm(2m,2)$ & $VO^\pm_{2m}(2)$ & for $m\geq 2$; affine polar graph & \cite[Sec. 3.3]{Brouwer2022}
\\ \hline
$Sp(2m,2)$ & $O^\pm(2m,2)$ & $NO^\pm_{2m}(2)$ & for $m\geq 3$; nonsingular points & \cite[Sec. 3.1.2]{Brouwer2022} \\
& $Sp(m,4)\rtimes C_2$ & $NO^\pm_{m+1}(4)$ & for $m\geq 4$ even; two classes of graphs & \cite[Sec. 3.1.4]{Brouwer2022}
\\ \hline
$P\Sigma L(2,q)$ & $C_{\frac{q+1}{2}} \rtimes C_{4e}$ & new strongly regular graph & $q=p^{2e}>9$, $p\equiv 3\pmod{4}$ & Section \ref{sec:linear}
\\ \hline\hline
$Sp(6,2)$ & $G_2(2)$ & $U_3(3)$-graph & $|V|=36$; subconstituent of the Hall--Janko graph & 
\\ \hline 
$Sp(8,2)$ & $S_{10}$ & generalized Johnson graph $J(10,3,1)$ & $|V|=120$ & \\
& $PSL(2,17)$ & $srg(136, 63, 30, 28)$ & $|V|=136$ & \cite{Bailey2022,Crnkovic2019}
\\ \hline
$P\Sigma L(2,9)$ & $S_5$ & Petersen graph & $P\Sigma L(2,9)$ has two conjugacy classes of $S_5$ &
\\ \hline
$HS$ & $M_{22}$ & $srg(176, 70, 18, 34)$ & parameter uniqueness & \cite{Crnkovic2024}, \cite{Degraer2008}, \cite[Thm. 5.2]{Goethals1970}
\\ \hline\hline
\end{tabular}
\end{table}

Complementing $\Gamma$ complements its associated two-graph and leaves both automorphism groups unchanged. The Goethals graph provides a related example outside Theorem~\ref{thm:main}. Its automorphism group is $S_7$, which is maximal in the index-three subgroup $PSU(3,5)\rtimes C_2$ of the full two-graph automorphism group $P\Gamma U(3,5)$. See Section~\ref{subsec:S7-graph}.
All the graphs in Table \ref{tab:main} are known, with the exception of the infinite family associated with $P\Sigma L(2,q)$; to the authors' knowledge these graphs are new, and the existence of these graphs is established in Section \ref{sec:linear}.  $ASp(2m,2)$ denotes the affine symplectic group, i.e.\ the semidirect product $G\cong \mathbb{F}_2^{2m}\rtimes Sp(2m,2)$.  The term \textit{parameter uniqueness} used in Table \ref{tab:main} means that the parameters determine the strongly regular graph up to isomorphism.


The paper is organized as follows. In Section~\ref{sec:prelim}, we will give all the necessary preliminaries on (finite) 2-transitive groups and two-graphs.  Section~\ref{sec:known-graphs} reviews the known infinite families of strongly regular graphs which appear in the paper, along with the associated groups, as well as the sporadic and exceptional examples.  In Section~\ref{sec:coboundary}, we explain the algebraic language of cocycles and coboundaries to represent graphs and two-graphs, while in Section~\ref{sec:cob-action}, we define the linear and affine actions of the automorphism groups on the switching class. Section~\ref{sec:linear} is devoted to the new infinite family of graphs with linear type two-graphs.  Finally, Sections~\ref{sec:nonexistence} and \ref{sec:uniqueness} contain proofs of the nonexistence of certain graphs, using deep results on the classification of maximal subgroups of almost simple groups.

\section{Preliminaries}\label{sec:prelim}
\subsection{Group actions and 2-transitivity}
A permutation group $G$ on a set $A$ is a subgroup of the symmetric group $Sym(A)$. An action of $G$ on $A$ may be regarded as a homomorphism from $G$ into the symmetric group. The image of an element $a\in A$ under the permutation
corresponding to $g\in G$ will be written $a^g$. We say that $G$ is \textit{transitive} in its action on $A$ if for all $a_1,a_2\in A$, there exists $g\in G$ with $a_1^g = a_2$. The relation $\sim$ on $A$ where $a_1\sim a_2$ if and only if $a_1^g = a_2$ for some $g\in G$ is an equivalence relation, whose equivalence classes are called \textit{orbits}. Therefore, $G$ is transitive on each of its orbits. For a positive integer $t$, a permutation group $G$ is \textit{$t$-transitive} on $A$ if its induced action on the set of ordered $t$-tuples of distinct elements of $A$ is transitive. The \textit{orbitals} of $G$ are the orbits of its action on $A\times A$, and the number of orbitals is called the \textit{rank} of $G$. For $|A|\geq1$, $G$ has rank at least 2, with equality if and only if $G$ is 2-transitive. Moreover, $G$ is 2-transitive if and only if the stabilizer of each element $a\in A$ is transitive on $A\setminus\{a\}$. 
We will not need the full classification of finite $2$-transitive groups: Table~\ref{tab:2transitive-groups} lists the $2$-transitive actions that occur in the present paper; for the full classification, we refer to Cameron~\cite{Cameron1981}.  For finite simple groups, we follow the notation used in the \textsc{Atlas} \cite{AtlasV3} (see also Wilson~\cite{Wilson2009}).  Extensions of groups will be written as follows: $H\times K$ will denote a \textit{direct product}; $H\rtimes K$ will denote a \textit{semidirect product} with normal subgroup $H$ and complement $K$.

\begin{table}[ht]
\centering
\caption{The $2$-transitive groups occurring in the paper.}
\label{tab:2transitive-groups}
\begin{tabular}{c|c|c|c}
Type & $G$ & Degree & Point stabilizer $G_v$\\
\hline
Affine polar
& $ASp(2m,2)$
& $2^{2m}$
& $Sp(2m,2)$
\\
Symplectic
& $Sp(2m,2)$
& $2^{2m-1}\pm2^{m-1}$
& $O^\pm(2m,2)$
\\
Linear
& $P\Sigma L(2,q)$
& $q+1$
& $P_1$
\\
Unitary
& $P\Gamma U(3,q)$
& $q^3+1$
& $P_1$
\\
Ree
& $Ree(q)\rtimes Aut(\mathbb F_q)$
& $q^3+1$
& $P$
\\
Sporadic
& $HS$
& $176$
& $PSU(3,5)\rtimes C_2$
\\
Sporadic
& $Co_3$
& $276$
& $McL\rtimes C_2$
\end{tabular}
\end{table}

\subsection{Two-graphs}
\begin{defn}[Two-graph]
A \textit{two-graph} is a pair $(V, T)$, where $T$ is a set of unordered triples of a vertex set $V$, such that every (unordered) quadruple from $V$ contains an even number of triples from $T$. The two-graph is called \textit{regular} if each pair of vertices is in a constant number of triples.
\end{defn}
\begin{defn}[Associated two-graph]
Given a graph $\Gamma = (V, E)$, the set of triples $T$ of the vertex set $V$, whose induced subgraph has an odd number of edges, forms a two-graph on the set $V$. The two-graph $\ome(\Gamma)=(V,T)$ is called the \textit{associated two-graph} of $\Gamma$.
\end{defn}

Let $\mathcal{T}=(V,T)$ be a two-graph. The \textit{descendant} of $\mathcal{T}$ with respect to the vertex $v\in V$ is the graph $\mathcal{T}_v=(V,E)$ where $E$ consists of pairs $\{w,z\}$ such that $\{v,w,z\}\in T$. By definition, $v$ is an isolated vertex of the descendant. It is not hard to see that $\ome(\mathcal{T}_v)=\mathcal{T}$, which shows that every two-graph can be represented as the associated two-graph of a graph.

\begin{defn}[Seidel switching \cite{Seidel1974}]
Given a graph $\Gamma=(V, E)$ and a subset $Y$ of the vertex set $V$, the operation of {\em switching} $\Gamma$ with respect to $Y$ consists of replacing all edges from $Y$ to $V\setminus Y$ by nonedges, and all such nonedges by edges, while leaving the edges within $Y$ or outside $Y$ unchanged.
\end{defn}
The subset $Y$ is called a \emph{switching set}. We say that a switching set $Y$ is \emph{regular of degree $d$} if the subgraph induced by $Y$ is $d$-regular.
Switching defines an equivalence relation on the set of all graphs with a given $n$-element vertex set, each equivalence class containing $2^{n-1}$ graphs (since switching with respect to a set and its complement yield the same graph, while switching with respect to two sets is the same as switching with respect to their symmetric difference). The equivalence classes are called \textit{switching classes}. It is well known that a pair of graphs are switching equivalent if and only if they have the same associated two-graph. Consequently, it is meaningful to speak of the switching class of a two-graph. 

We now give a few important results regarding the parameters of two-graphs and related strongly regular graphs.

\begin{res}[{\cite[Proposition 1.1.2]{Brouwer2022}}] \label{res:rels-switching1}
Let $\Gamma$ be a strongly regular graph with parameters $(v,k,\lambda,\mu)$ and spectrum $k^1 r^f s^g$. Let $\Delta$ be a strongly regular graph of valency $\ell>k$ switching equivalent to $\Gamma$. Then:
\begin{enumerate}
    \item $\Delta$ has spectrum $\ell^1 r^{f-1} s^{g+1}$,
    \item $\frac{v}{2}=k-s=\ell-r$,
    \item $k-r=2\mu$,
    \item $\frac{v}{2}=2k-\lambda-\mu$.
    \item Any switching set from $\Gamma$ to $\Delta$ has size $\frac{1}{2}v$ and is regular of degree $k-\mu$.
\end{enumerate}
\end{res}

\begin{res}[{\cite[Proposition 10.3.3]{Brouwer2012}}] \label{res:rels-switching2}
Let $\Gamma$ be a strongly regular graph with parameters $(v,k,\lambda,\mu)$ associated with a regular two-graph. Then either:
\begin{enumerate}
    \item the graph $\Gamma$ is switched into a strongly regular graph with the same parameters if and only if every vertex outside the switching set $S$ is adjacent to half of the vertices of $S$,
    \item the graph $\Gamma$ is switched into a strongly regular graph with parameters $(v,k+c,\lambda+c,\mu+c)$ where $c =\frac{v}{2}-2\mu$ if and only if the switching set $S$ has size $\frac{v}{2}$ and is regular of degree $k-\mu$.
\end{enumerate}
\end{res}

\begin{res}[{\cite[Theorems 10.3.1 and 10.3.2]{Brouwer2012}}] \label{res:rels-twograph} 
Let $\mathcal{T}$ be a nontrivial regular two-graph on $v$ vertices and regularity $a$.
\begin{enumerate}[(i)]
\item For any descendant $\mathcal{T}_v$, $\mathcal{T}_v-v$ is a $srg\left(v-1,a,\frac{3a-v}{2},\frac{a}{2}\right)$.
\item Any $k$-regular graph $\Gamma$ in the switching class of $\mathcal{T}$ is strongly regular, and $k$ is a root of the polynomial
\[X^{2} - {\left(\frac{v}{2} + a\right)} \, X + \frac{a}{2} {\left(v - 1\right)} ,\]
and
\[\lambda = k+\frac{a-v}{2}, \qquad \mu=k-\frac{a}{2}.\]
\end{enumerate}
\end{res}
For the strong regularity assertion in (ii), count triples through a pair of vertices. If the pair has $c$ common neighbours, it belongs to $v-2k+2c$ triples of the two-graph when it is adjacent and to $2(k-c)$ when it is nonadjacent. Thus $c$ depends only on adjacency. The same count gives the formulas for $\lambda$ and $\mu$, and $k(k-\lambda-1)=(v-k-1)\mu$ gives the polynomial.

\subsection{2-transitive two-graphs}\label{sec:construction}

In this section, we review Taylor's classification of finite two-graphs with 2-transitive automorphism groups. 

Trivially, the only graphs with 2-transitive automorphism groups are complete graphs or their complements, since edges cannot be mapped to non-edges and vice-versa. However, regular two-graphs may have 2-transitive automorphism groups; see \cite[Theorem 6.1]{Taylor1977}. Let $G$ be a (finite) group acting 2-transitively on a set $V$ of vertices and consider a (regular) two-graph $\mathcal{T}$ on $V$ such that $G\leq \Aut(\mathcal{T})$. By \cite{Taylor1992}, we know a complete classification of such two-graphs, by their automorphism groups.
 \begin{res}[{\cite[Theorem 1]{Taylor1992}}] \label{res:Taylorgrs}
     If $(V,T)$ is a nontrivial two-graph admitting an automorphism group $G$ acting $2$-transitively on $V$, then one of the following occurs:
     \begin{enumerate}
         \item $G$ has a normal subgroup $N$, and letting $n=|V|$:
         \begin{itemize}
             \item $N\cong PSL(2,q)$, $q\equiv1\pmod4$, $n=q+1$;
             \item $N\cong PSU(3,q)$, $q$ odd, $n=q^3+1$;
             \item $N\cong Ree(q)$, $q=3^{2e+1}$, $n=q^3+1$;
             \item $N\cong Sp(2m,2)$, $m\geq3$, $n=2^{2m-1}\pm2^{m-1}$;
             \item $N\cong HS$, $n=176$;
             \item $N\cong Co_3$, $n=276$.
         \end{itemize}
         \item $G$ is of affine type. 
     \end{enumerate}
 \end{res}

\begin{res}[{\cite[Theorem 2]{Taylor1992}}] \label{res:Taylorconstr}
     For each permutation action in the classification there is a unique pair of complementary nontrivial two-graphs. The full automorphism groups in the respective actions are as follows.
     \begin{enumerate}
     \item
         \begin{description}
             \item[Linear type] $G\cong P\Sigma L(2,q)$, for $q\equiv 1\pmod{4}$;
             \item[Unitary type] $G \cong P\Gamma U(3,q)$, for $q\geq 5$ odd;
             \item[Ree type] $G\cong Ree(q)\rtimes \Aut(\mathbb F_q)$, for $q=3^{2e+1}$, $e\geq 1$;
             \item[Symplectic type] $G\cong Sp(2m,2)$, $m\geq 3$;
             \item[Sporadic type] $G\cong HS$ or $G\cong Co_3$.
         \end{description}
         \item 
         \begin{description}
             \item[Affine polar type] The automorphism group is the semidirect product $\mathbb{F}_2^{2m}\rtimes Sp(2m,2)$, where $\mathbb{F}_2^{2m}$ is the additive group of the vector space of dimension $2m$ over $\mathbb{F}_2$.
         \end{description}
     \end{enumerate}
\end{res}

The symplectic groups $Sp(2,2)\cong S_3$ and $Sp(4,2)\cong S_6$ act 3-transitively on $3$ and $6$ points; hence, the associated two-graphs are trivial. The 2-transitive actions of degree 10 of $Sp(4,2)$ and $P\Sigma L(2,9)$ are equivalent, and the corresponding two-graphs are the same. In their 2-transitive actions of degree 28, $P\Gamma U(3,3)$ and $Ree(3)\cong P\Gamma L(2,8)$ are (maximal) subgroups of $PSp(6,2)$. The two-graphs obtained from these three groups are isomorphic. 

The 2-transitive automorphism group $G$ does not define the regular two-graph $\mathcal{T}$ uniquely, only up to taking the complement. If $v=|V|$ is the number of vertices, then one has two choices $a,a'$ for the regularity of $\mathcal{T}$; $a+a'=v-2$. Typically, one uses the values for $a$ and $a'$ as listed in Table \ref{tab:regularity-properties}.

\begin{table}
    \caption{The parameters $v$, $a$ and $a'$ of 2-transitive two-graphs}
    \label{tab:regularity-properties}
    \centering
    \begin{tabular}{lllll}
        Type &  $G$ & $v$ & $a$ & $a'$ \\
        \hline\hline
        Affine & $ASp(2m,2)$ & $2^{2m}$ & $2^{2m-1}-2$ & $2^{2m-1}$ \\
        Symplectic & $Sp(2m,2)$ & $2^{2m-1}\pm 2^{m-1}$ & $2^{2m-2}\pm 2^{m-1}-2$ & $2^{2m-2}$ \\
        Linear & $P\Sigma L(2,q)$ & $q+1$ & $(q-1)/2$ & $(q-1)/2$ \\
        Unitary & $P\Gamma U(3,q)$ & $q^{3}+1$ & $(q-1)(q^2+1)/2$ & $(q+1)(q^2-1)/2$ \\
        Ree & $Ree(q)$ & $q^{3}+1$ & $(q-1)(q^2+1)/2$ & $(q+1)(q^2-1)/2$ \\
        Sporadic & $HS$ & 176 & 72 & 102 \\
         & $Co_3$ & 276 & 162 & 112 \\
        \hline\hline
    \end{tabular}
\end{table}

\section{Known graphs and their two-graphs in Table \ref{tab:main}} \label{sec:known-graphs}

In this section, we present the known graphs in Table~\ref{tab:main} and the related example of the Goethals graph. At the end of the section, we identify their associated two-graphs (Proposition~\ref{prop:table-two-graphs}).

\subsection{Affine polar graphs \texorpdfstring{$VO^\pm_{2m}(2)$}{VO{2m}(2)}}\label{affinepolar}

Let $m\geq2$. Let $Q$ be a nondegenerate quadratic form of type $\pm 1$ in the vector space $V=\mathbb{F}_2^{2m}$, and let $\langle x,y \rangle =Q(x+y)+Q(x)+Q(y)$ be the associated symplectic bilinear form. In $VO^\pm_{2m}(2)$, two different vectors $x,y$ are adjacent when $Q(x+y)=0$. This yields a strongly regular graph with parameters
\begin{align*}
v&=2^{2m}, & k&=(2^{m}\mp 1)(2^{m-1}\pm 1), \\
\lambda &=2(2^{m-1}\mp 1)(2^{m-2}\pm 1) , & \mu&=2^{m-1}(2^{m-1}\pm 1).
\end{align*}
The full automorphism group is $Aut(VO^{\pm}_{2m}(2))\cong\mathbb F_2^{2m}\rtimes O^{\pm}(2m,2)$, consisting of the semidirect product of the orthogonal group and the additive translation group, as in the first row of Table \ref{tab:main}. See \cite[Section 3.3]{Brouwer2022} for more information.

\subsection{Graphs \texorpdfstring{$NO^\pm_{2m}(2)$}{NO\textpm{}(2m,2)} on nonsingular points over \texorpdfstring{$\mathbb{F}_2$}{GF(2)}} \label{subsec:NO-2m-2}
Let $m\geq3$. The graph $NO^{\pm}_{2m}(2)$ is the graph whose vertex set is $PG(2m-1,2)\setminus Q^{\pm}(2m-1,2)$, and two vertices are adjacent if and only if the corresponding points lie on a line tangent to the quadric $Q^{\pm}(2m-1,2)$. $NO^{\pm}_{2m}(2)$ is a strongly regular graph with the following parameters: 
\begin{align*}
v&=2^{2m-1}\mp2^{m-1}, & k&=2^{2m-2}-1, \\
\lambda &=2^{2m-3}-2, & \mu&=2^{2m-3}\pm2^{m-2}.
\end{align*}
The automorphism group of $NO^{\pm}_{2m}(2)$ is the orthogonal group $O^{\pm}(2m,2)$. See \cite[Section 3.1.2]{Brouwer2022} for more information on this graph. 

\subsection{Graphs \texorpdfstring{$NO^{\pm}_{2m+1}(4)$}{NO\textpm{}(2m+1,4)} on nonsingular points of one type} \label{subsec:no-2m+1-4}
Let $Q(2m,4)$ be a nonsingular parabolic quadric in $PG(2m,4)$. Let $\mathcal{H}^{\pm}$ be the set of nontangent hyperplanes intersecting the quadric in a nonsingular quadric of type $\pm1$. The graph $NO^{\pm}_{2m+1}(4)$ is the graph whose vertex set is $\mathcal{H}^{\pm}$, and two vertices $H_1$ and $H_2$ are adjacent if and only if $H_1\cap H_2\cap Q(2m,4)$ is degenerate. The graph $NO^{\pm}_{2m+1}(4)$ is strongly regular, and its automorphism group is $Aut(NO^{\pm}_{2m+1}(4))\cong Sp(2m,4)\rtimes C_2$.  This graph is described in \cite[Section 3.1.4]{Brouwer2022}.

In \cite{Nagy2024}, the authors proved that the two families $NO^\mp_{4m}(2)$ and $NO^{\pm}_{2m+1}(4)$ are switching equivalent. The parameters of $NO^{\pm}_{2m+1}(4)$ are:
\begin{align*}
v&=2^{4m-1}\pm2^{2m-1}, & k&=(2^{2m-2}\pm1)(2^{2m}\mp1), \\
\lambda &=2^{4m-3}\pm3\cdot2^{2m-2}-2, & \mu&=2^{4m-3}\pm2^{2m-1}.
\end{align*}
For a nondegenerate quadratic form $\Theta$ of type $\pm$ on $\mathbb F_2^{2d}$, define the symplectic two-graph $\mathcal X^\pm_{2d}=(X,T)$ by
\begin{align*}
    X & =\{a\in \mathbb{F}_{2}^{2d} \mid \Theta(a)=1\},\\
    T & =\{\{a,b,c\} \mid a,b,c\in X\text{ distinct},\
                       \langle a,b\rangle+\langle a,c\rangle+\langle b,c\rangle=0\},
\end{align*}
where $\langle\cdot,\cdot\rangle$ is the polar form of $\Theta$. For $d=2m$, this is the two-graph common to $NO^\pm_{4m}(2)$ and $NO^\mp_{2m+1}(4)$.

\subsection{Sporadic and exceptional graphs}

\subsubsection{The $S_{5}$-invariant graph: The Petersen graph and its complement} \label{subsec:petersen}
The \textit{Petersen graph} with parameters $(10,3,0,1)$ and its complement with parameters $(10,6,3,4)$ are the only primitive strongly regular graphs on 10 vertices. The Petersen graph is isomorphic to the tangent graph $NO^-_4(2)$, while its complement is isomorphic to both $NO^+_3(4)$ and the triangular graph $T(5)$; see Sections \ref{subsec:NO-2m-2} and \ref{subsec:no-2m+1-4}. The automorphism groups of both graphs are isomorphic to $S_5\cong O^-(4,2) \cong Sp(2,4)\rtimes C_2$. Their associated two-graphs are complementary and isomorphic, with 60 unordered triples each and full automorphism group $P\Sigma L(2,9) \cong Sp(4,2) \cong S_6$.

\subsubsection{The $G_2(2)$-invariant graph} \label{subsec:G22-graph}
The group $G_2(2)\cong PSU(3,3)\rtimes C_2$ has a primitive permutation representation on $36$ vertices. An $srg(36,14,4,6)$ is known to exist with automorphism group $G_2(2)$ and point stabilizer $PGL(2,7)$. This graph is also called the $U_3(3)$-graph; see \cite[Section 10.14]{Brouwer2022} for further details. Although the graph is not determined by its parameters alone, it is the unique graph with these parameters that admits a transitive rank-3 action of $G_2(2)$ of degree 36.

\subsubsection{Goethals' $S_{7}$-invariant graph} \label{subsec:S7-graph}
There is a unique strongly regular graph with parameters $(v, k, \lambda, \mu) = (126, 50, 13, 24)$. The full automorphism group is $S_7$. Existence is due to Goethals; uniqueness is due to Coolsaet and Degraer (see \cite[Section 10.42]{Brouwer2022}, \cite{Coolsaet2008} and \cite{Haemers2002}). 
Its associated two-graph is the Hermitian two-graph on 126 vertices, with full automorphism group $P\Gamma U(3,5)$. Every $S_7$ subgroup of this group lies in an index-three subgroup $PSU(3,5)\rtimes C_2$, since the image of $S_7$ in the outer group $S_3$ has order at most two. The Goethals graph therefore lies outside the scope of Theorem~\ref{thm:main}.


\subsubsection{The $S_{10}$-invariant graph} \label{subsec:S10-graph}

When $H=S_{10}$, we obtain the strongly regular generalized Johnson graph $J(10,3,1)$. Recall that the generalized Johnson graphs are the graphs $J(n, k, m)$ whose vertices are the $k$ subsets of $\{1, 2, \ldots, n\}$, with two vertices $A, B$ joined by an edge if and only if $|A\cap B|=m$. The graphs $J(n, k, 1)$ are typically not strongly regular, but Mathon and Klin showed that $J(10,3,1)$ is an exception (see \cite[Section 1.3.6]{Brouwer2022}).  In particular, $J(10,3,1)$ is strongly regular with parameters $(120,63,30,36)$, and is therefore cospectral to $NO^+_8(2)$, but is not isomorphic to it since the automorphism group is $S_{10}$.

\subsubsection{The $PSL(2,17)$-invariant graph} \label{subsec:PSL217-graph}

Let $G=Sp(8,2)$, in its 2-transitive action of degree 136. Then the group $H=PSL(2,17)$ is a transitive maximal subgroup. In \cite{Bailey2022} the first author and his students obtained a strongly regular graph with parameters $(136,63,30,28)$ and full automorphism group $PSL(2,17)$ (acting primitively with rank 12), as a result of a computer search using the GRAPE package for GAP \cite{GAP,GRAPE}; the existence of such a graph was also established independently in \cite{Crnkovic2019}.  The graph is cospectral with but is not isomorphic to $NO^-_8(2)$.

\subsubsection{The $M_{22}$-invariant graph} \label{subsec:M22-graph}

Let $G=HS$ be the Higman--Sims group in its action of degree 176. Its maximal Mathieu subgroup $M_{22}$ is the automorphism group of an $srg(176,70,18,34)$. This graph was constructed by Goethals and Seidel \cite[Theorem 5.2]{Goethals1970}; see also \cite[Section 10.51]{Brouwer2022} and \cite{Crnkovic2024} for further details. The graph is uniquely determined by its parameters \cite{Degraer2008}. This graph is an induced subgraph of the strongly regular McLaughlin graph $srg(275,112,30,56)$. Indeed, $M_{22}$ is a maximal subgroup of the McLaughlin group $McL$ with orbit lengths $22,77$ and $176$.

We note that the Higman--Sims group $HS$ has a maximal subgroup isomorphic to $S_8$. Starting from the Higman symmetric design on 176 points (which has $HS$ as its automorphism group) Brouwer \cite{Brouwer1982} constructed a $srg(176,49,12,14)$ with this action of $S_8$ as its automorphism group (see also \cite[Section 10.50]{Brouwer2022}.  However, this action is intransitive, so this graph cannot arise in our classification.  Furthermore, its associated two-graph is not regular (this can be seen from the eigenvalues of the graph) and thus cannot be the Higman--Sims two-graph.

\subsection{Identifying the associated two-graphs}

\begin{prop}\label{prop:table-two-graphs}
    The strongly regular graphs described in this section have the following associated two-graphs:
    \begin{itemize}
       \item[(i)] The affine polar graphs $VO^{\pm}_{2m}(2)$ have the affine two-graph with full automorphism group $ASp(2m,2)$.
       \item[(ii)] The graph $NO^{\mp}_{2m}(2)$ and, when $m$ is even, the graph $NO^{\pm}_{m+1}(4)$ have the symplectic two-graph $\mathcal X^{\mp}_{2m}$.
        \item[(iii)] The Petersen graph and its complement have isomorphic linear two-graphs on 10 vertices, with full automorphism group $P\Sigma L(2,9)\cong Sp(4,2)\cong S_6$.
        \item[(iv)] The complement of the $U_3(3)$-graph has the symplectic two-graph $\mathcal T^+$ on 36 vertices of regularity $18$, with full automorphism group $Sp(6,2)$. The graph itself has the complementary two-graph.
        \item[(v)] The Goethals graph has the Hermitian two-graph on 126 vertices, whose full automorphism group is $P\Gamma U(3,5)$.
        \item[(vi)] The graph $J(10,3,1)$ has the symplectic two-graph on 120 vertices, whose full automorphism group is $Sp(8,2)$.
        \item[(vii)] The $PSL(2,17)$-invariant  $srg(136,63,30,28)$ has the symplectic two-graph on 136 vertices, whose full automorphism group is $Sp(8,2)$.
       \item[(viii)] The $M_{22}$-invariant  $srg(176,70,18,34)$ has the Higman-Sims two-graph on 176 vertices, whose full automorphism group is $HS$.
    \end{itemize}
\end{prop}

\begin{proof}
Part (i) follows from the affine model in Section~\ref{sec:coboundary}, and (ii) follows from \cite[Proposition~4.1 and Theorem~4.7]{Nagy2024}. Part (iii) follows from the Petersen graph's polar model, and (v) follows from \cite{Haemers2002,Taylor1992}.

For (vi), realize the symplectic space over $\mathbb F_2$ as the even-weight subspace of $\mathbb F_2^{10}$ modulo the all-one vector, with the form induced by the dot product. Fix a singleton $R\subseteq\{1,\dots,10\}$. The hyperbolic quadratic form
\[
Q_R([x])=\tfrac12\operatorname{wt}(x)+\sum_{i\in R}x_i\pmod2
\]
is well defined on the quotient. Its nonsingular vectors are precisely the $120$ vectors $u_T=[\mathbf1_{T\mathbin{\triangle}R}]$ for triples $T$. In fact, each coset admits a representative $\mathbf1_{T\mathbin{\triangle}R}$ with $T$ of odd size at most five, and substitution gives $Q_R(u_T)=1$ exactly when $|T|=3$. For distinct triples $T,U$ adjacency in the polar graph $NO^+_8(2)$ on these vectors is given by
\[
1+\langle u_T,u_U\rangle=1+Q_R(u_T+u_U)=|T\cap U|+|R\cap T|+|R\cap U|\pmod2 .
\]
Since $|T\cap U|\in\{0,1,2\}$, the first term gives adjacency in $J(10,3,1)$. The remaining terms constitute a switching cut. This proves (vi).

The claims in (iv), (vii), and (viii) can be verified by computation in GAP \cite{GAP}. 
\end{proof}
We conclude this section by remarking that parts (vi) and (vii) of Proposition~\ref{prop:table-two-graphs} identify new switching equivalences with $NO^+_8(2)$ and $NO^-_8(2)$, respectively.

\section{The coboundary operator} \label{sec:coboundary}

Taylor \cite{Taylor1992} introduced the notion of the coboundary operator to study graphs and their two-graphs in an algebraic language. To provide a uniform algebraic framework for describing switching classes, we will adopt this concept.

In this section, $V$ is a nonempty finite set, and $\mathbb{F}_2 = \{0,1\}$ is the finite field of order 2. For a nonnegative integer $k$, a \textit{$k$-cochain} is a totally symmetric mapping $f:V^{k+1}\to \mathbb{F}_2$ such that $f(x_0,\ldots,x_k)=0$ whenever two of $x_0,\ldots,x_k$ are equal. The set $V^{(k,*)}$ of $k$-cochains is an $\mathbb{F}_2$-linear space of dimension $\binom{|V|}{k+1}$. Their support 
\[\supp(f)=\{\{x_0,\ldots,x_k\} \mid f(x_0,\ldots,x_k)=1\}\]
uniquely defines the $k$-cochains. Following Taylor \cite{Taylor1977}, we define the \textit{coboundary operator}
\[\partial: V^{(k,*)} \to V^{(k+1,*)}\]
by
\[\partial f(x_0,\ldots,x_{k+1})=\sum_{i=0}^{k+1} f(x_0,\ldots,x_{i-1},x_{i+1},\ldots,x_{k+1}).\]
A $k$-cochain of the form $\partial f$ is called a \textit{$k$-coboundary}, and a $k$-cochain $f$ such that $\partial f = 0$ is called a \textit{$k$-cocycle}. An easy calculation shows that $\partial^2 f = 0$, so every coboundary is a cocycle. For $k\geq1$, every $k$-cocycle is also a $k$-coboundary. To see this, choose $v \in V$ and define the contraction $\Delta_v: V^{(k+1,*)} \to V^{(k,*)}$ by $(\Delta_v f)(x_0,\ldots,x_k) = f(x_0,\ldots,x_k,v)$ for every $k\geq 0$. For $f \in V^{(k,*)}$ with $k\geq1$ we have
\[f=\Delta_v(\partial f) + \partial(\Delta_v f).\]
In particular, if $\partial f = 0$, then $f = \partial(\Delta_v f)$.

Now observe that a graph with vertex set $V$ can be regarded as a 1-cochain $f: V^{2} \to \mathbb{F}_2$, where $f(x,y) = 1$ if and only if $x$ and $y$ are adjacent. A 2-cochain $\tau$ defines a two-graph $\mathcal{T}=(V,\supp(\tau))$ precisely when $\partial\tau=0$. For such a $\tau$, the 1-cochain $\Delta_v\tau$ is the descendant at $v$. More precisely, the graph of the 1-cochain $f$ has $\mathcal{T}$ as its associated two-graph if and only if $\partial f=\tau$. In other words, the switching class of $\tau$ consists of the 1-cochains $f$ such that $\partial f=\tau$.

Let $G$ be a group acting on $V$. Then $G$ has an induced action on cochains:
\[f^g(x_0,\ldots,x_k)=f(x_0^{g^{-1}},\ldots,x_k^{g^{-1}}).\]
This action is $\mathbb{F}_2$-linear, and it is equivalent to the natural action of $G$ on graphs and two-graphs with vertex set $V$. The \textit{automorphism group} of $f$ is the set $\Aut(f)$ of those permutations $\pi$ of $V$ for which $f=f^\pi$. 

Evidently, $\partial (f^g) =(\partial f)^g$, which implies that $G$ acts on cocycles. The following proposition demonstrates the action of $G$ on switching classes. 
\begin{prop} \label{pr:W-tau-action}
Let $\tau\in V^{(k+1,*)}$ be a cocycle and define 
\[\mathcal{W}(\tau)=\{h\in V^{(k,*)} \mid \partial h=\tau\}.\]
Then $\mathcal{W}(0)$ and $\mathcal{W}(0) \cup \mathcal{W}(\tau)$ are $\mathbb{F}_2$-linear spaces and $\Aut(\tau)$ induces a linear action on them.
\end{prop}
\begin{proof}
We only have to show that any $g\in \Aut(\tau)$ preserves $\mathcal{W}(\tau)$. This follows from $\partial(h^g)=(\partial h)^g$, $0^g=0$, and $\tau^g=\tau$. 
\end{proof}

We illustrate this method on the class of affine polar graphs $VO^\pm_{2m}(2)$; see \cite[Section 3.3]{Brouwer2022} and Section \ref{affinepolar}. $VO^{\pm}_{2m}(2)$ is a strongly regular graph with parameters
\begin{align*}
v&=2^{2m}, & k&=(2^{m}\mp 1)(2^{m-1}\pm 1), \\
\lambda &=2(2^{m-1}\mp 1)(2^{m-2}\pm 1) , & \mu&=2^{m-1}(2^{m-1}\pm 1).
\end{align*}
The vertex set of $VO^\pm_{2m}(2)$ is $\mathbb{F}_2^{2m}$, and adjacency is defined by $Q(x+y)=0$, where $Q(x)$ is a nondegenerate quadratic form of type $\pm 1$. $Q(x)$ linearizes to the symplectic bilinear form 
\[\langle x,y\rangle =Q(x+y)+Q(x)+Q(y).\]
The complement graph $\overline{VO^\pm_{2m}(2)}$ has parameters
\[(2^{2m},2^{m-1}(2^m\mp 1),2^{m-1}(2^{m-1}\mp 1),2^{m-1}(2^{m-1}\mp 1)),\]
and can be given by the cochain
\[f(x,y)=Q(x+y)=Q(x)+Q(y)+\langle x,y \rangle.\]
Applying the coboundary operator, we get the 2-cocycle
\[\partial f(x,y,z)=\langle x,y \rangle+\langle x,z \rangle+\langle y,z \rangle\]
of the associated two-graph $\mathcal{T}$. The descendant $\mathcal{T}_0$ is the graph on the set of vertices $\mathbb{F}_2^{2m}$, where $x$ and $y$ are adjacent whenever $\langle x,y\rangle =1$. By deleting the isolated vertex $0$ of $\mathcal{T}_0$, we obtain the complement of the symplectic graph over $\mathbb{F}_2$, which is strongly regular with parameters $(2^{2m}-1,2^{2m-1},2^{2m-2},2^{2m-2})$; see \cite[Section 2.5]{Brouwer2022}. 

The automorphism group of the two-graph $\mathcal{T}$, that is, the stabilizer of $\partial f$ is the semidirect product $G=\mathbb{F}_2^{2m}\rtimes Sp(2m,2)$  with its 2-transitive action on $\mathbb{F}_2^{2m}$. The automorphism group of the affine polar graph $VO^\pm_{2m}(2)$ is the semidirect product $H=\mathbb{F}_2^{2m}\rtimes O^\pm(2m,2)$. $H$ is a transitive maximal subgroup of $G$.

\section{Linear and affine actions on the switching class} \label{sec:cob-action}

Let $W$ be a vector space over a field $K$. An \textit{affine combination} of $w_0,\ldots,w_m\in W$ is $c_0w_0+\cdots+c_mw_m$, where $c_0,\ldots,c_m\in K$ with $c_0+\cdots+c_m=1$. The set of all affine combinations of $w_0,\ldots,w_m$ is the \textit{affine subspace} $W'$ spanned by $w_0,\ldots,w_m$. Its dimension is one less than the size of a smallest affine generating set. Any affine subspace is obtained as a translate of a linear subspace. Invertible affine maps $x\mapsto xA+b$ permute the affine subspaces, where $b\in W$ and $x\mapsto xA$ is invertible and linear. Affine maps may be represented by matrices of the form
\begin{align} \label{eq:affine-matrix}
\begin{bmatrix} A & 0\\b & 1 \end{bmatrix}.
\end{align}
This representation allows us to compute fixed points and invariant subspaces of affine linear maps using linear algebra tools. 

\begin{prop} \label{prop:bracket-actions}
Let $\tau$ be a $2$-coboundary, $\mathcal{T}=(V,T)$ the associated two-graph, and $G=\Aut(\mathcal{T})$. Fix a $1$-cochain $s\in V^{(1,*)}$ such that $\partial s=\tau$. Define
\[\mathcal{W}(\tau)=\{h\in V^{(1,*)} \mid \partial h=\tau\}.\]
For $v\in V$, write
\[\mathcal{V}(v)= \{f\in V^{(0,*)} \mid f(v)=0\}. \]
For $g\in G$ and $f\in \mathcal{V}(v)$, define
\begin{align}
f^{[g]}(x)&=f(x^{g^{-1}})+f(v^{g^{-1}}), \label{eq:square-bracket-action}\\
f^{\{g\}}(x)&=f(x^{g^{-1}})+f(v^{g^{-1}})+s(x^{g^{-1}},v^{g^{-1}})+s(x,v). \label{eq:curly-bracket-action}
\end{align}
Then:
\begin{enumerate}[(i)]
\item $f^{[g]}$ is a linear action of $G$ on $\mathcal{V}(v)$;
\item $f^{\{g\}}$ is an affine linear action of $G$ on $\mathcal{V}(v)$; 
\item the map $u:f\mapsto \partial f+s$ is an isomorphism between $\mathcal{V}(v)$ with the affine action and $\mathcal{W}(\tau)$ with the natural action.
\end{enumerate}
\end{prop}
\begin{proof}
(i) is straightforward. By $\partial s=\tau=\tau^g=\partial s^g$, we have
\begin{align} \label{eq:s-sg-id}
s(x,y)+s(x,v)+s(y,v) = s^g(x,y)+s^g(x,v)+s^g(y,v)
\end{align}
for all $x,y \in V$. We obtain (ii) from
\begin{align*}
(f^{\{g\}})^{\{h\}}(x) &= f^{\{g\}}(x^{h^{-1}}) + f^{\{g\}}(v^{h^{-1}}) + s(x^{h^{-1}},v^{h^{-1}}) + s(x,v) \\
&= f(x^{h^{-1}g^{-1}}) + f(v^{g^{-1}}) + s(x^{h^{-1}g^{-1}},v^{g^{-1}}) + s(x^{h^{-1}},v) \\
& \qquad + f(v^{h^{-1}g^{-1}}) + f(v^{g^{-1}}) + s(v^{h^{-1}g^{-1}},v^{g^{-1}}) + s(v^{h^{-1}},v)\\
& \qquad + s(x^{h^{-1}},v^{h^{-1}}) + s(x,v) \\
&= f(x^{h^{-1}g^{-1}}) + f(v^{h^{-1}g^{-1}}) + s^g(x^{h^{-1}},v) + s^g(v^{h^{-1}},v) \\
& \qquad + s(x^{h^{-1}},v) + s(x^{h^{-1}},v^{h^{-1}}) + s(v^{h^{-1}},v) + s(x,v) \\
&\stackrel{\eqref{eq:s-sg-id}}{=} f(x^{h^{-1}g^{-1}}) + f(v^{h^{-1}g^{-1}}) + s^g(x^{h^{-1}},v^{h^{-1}}) + s(x,v)\\
&=f(x^{(gh)^{-1}}) + f(v^{(gh)^{-1}}) + s(x^{(gh)^{-1}},v^{(gh)^{-1}}) + s(x,v)\\
&=f^{\{gh\}}(x).
\end{align*}
We show (iii). Since $\partial (u(f))=\partial s=\tau$, $u$ maps $\mathcal{V}(v)$ to $\mathcal{W}(\tau)$. If $h\in \mathcal{W}(\tau)$, then $\partial(h-s)=0$, and $h-s=\partial f$ with $f\in V^{(0,*)}$. Replacing $f$ with $f+f(v)$, we can assume that $f(v)=0$; therefore, $u$ is surjective. If $u(f_1)=u(f_2)$, then $\partial(f_1-f_2)=0$, which implies $f_1=f_2$ by $f_1(v)=f_2(v)=0$. It remains to show that $u$ is an isomorphism between $G$-sets. Fix $f\in \mathcal{V}(v)$, $g\in G$ and define 
\[\varphi(x)=f^{\{g\}}(x)=f^g(x)+f^g(v)+s^g(x,v)+s(x,v).\]
Then,
\begin{align*}
(\partial \varphi+s)(x,y) &= f^g(x)+f^g(y)+s^g(x,v)+s^g(y,v)+s(x,v)+s(y,v)+s(x,y) \\
&\stackrel{\eqref{eq:s-sg-id}}{=} f^g(x)+f^g(y)+s^g(x,y)\\
&=(\partial f+s)^g(x,y).
\end{align*}
This proves $u(f^{\{g\}})=u(f)^g$. 
\end{proof}

We will call the actions \eqref{eq:square-bracket-action} and \eqref{eq:curly-bracket-action} the \textit{square bracket action} and the \textit{curly bracket action}. When the element $v\in V$ is fixed, the natural choice for the 1-cochain $s$ is the contraction $s(x,y)=\Delta_v\tau(x,y)=\tau(x,y,v)$. In this case, $s(x,v)=0$, $s(x^{g^{-1}},v^{g^{-1}}) = \tau(x^{g^{-1}},v^{g^{-1}},v) = \tau^g(x,v,v^g) = \tau(x,v,v^g)$. The curly bracket action is 
\begin{align} \label{eq:curly-short}
f^{\{g\}}(x)=f^g(x)+f^g(v)+\tau(x,v,v^g).
\end{align}

We finish this section with an invariant of the curly bracket action for a regular two-graph of even order and even regularity. With $s=\Delta_v\tau$, define the sign of $f\in \mathcal{V}(v)$ by
\begin{align} \label{eq:sign-def}
sgn(f)=\sum_{x\in V} f(x).
\end{align}
For any $g\in G$, equation~\eqref{eq:curly-short} gives
\[
sgn(f^{\{g\}})=sgn(f)+|V|f(v^{g^{-1}})+\sum_{x\in V}\tau(x,v,v^g).
\]
The middle term vanishes in $\mathbb F_2$ because $|V|$ is even. The last term is zero if $v^g=v$ and is the regularity modulo two otherwise. Hence $sgn$ is an invariant, and it is additive by definition.

\section{Linear type two-graphs and their graphs} \label{sec:linear}

In this section, we classify the vertex-transitive strongly regular graphs in the switching classes of the linear-type two-graphs with full automorphism group $P\Sigma L(2,q)$. For $q=9$ the graphs are the Petersen graph and its complement. For $q>9$ they occur precisely when $q=p^{2e}$ and $p\equiv3\pmod4$. In each of those cases a fixed switching class contains precisely two isomorphism types, represented by graphs isomorphic to complements of one another. We determine the feasible parameter sets, classify the relevant transitive maximal subgroups $H$ (Section~\ref{sec:psigmal-trsubgrs}), and construct graphs from the fixed points of the curly bracket action of $H$ (Section~\ref{sec:vertex-transitive-linear-srg}).

Strongly regular graphs with these parameters are already known to exist by the construction in Section~\ref{sec:folklore-construction}. Its graphs are intransitive for $q>9$. To our knowledge, the vertex-transitive family for $q>9$ is new.

Let $p$ be a prime, $f$ a positive integer, and $q=p^f$. We represent the projective general linear group $PGL(2,q)$ as the group of fractional linear maps acting on $V=\mathbb{F}_q\cup \{\infty\}$. For $a,b,c,d \in \mathbb{F}_q$, $ad-bc\neq 0$, we denote by $\tmat{a}{b}{c}{d}$ the fractional linear mapping 
\[z\mapsto z^{\tmat{a}{b}{c}{d}}=\frac{az+b}{cz+d}.\]
The projective special linear group $PSL(2,q)$ consists of the maps $\tmat{a}{b}{c}{d}$ with $ad-bc\in S_q$, where $S_q$ is the set of nonzero squares in $\mathbb{F}_q$. The field automorphisms $V\mapsto x^{p^i}$ act on $V$ as well. The groups $P\Gamma L(2,q)$ and $P\Sigma L(2,q)$ are the closure groups of $PGL(2,q)$ and $PSL(2,q)$ with the group of field automorphisms. If $q$ is even, then $PGL(2,q)=PSL(2,q)$, while for odd $q$, $PSL(2,q)$ is a normal subgroup of index $2$ in $PGL(2,q)$. 

The projective general linear group $PGL(2,q)$ acts $3$-transitively on $V$; hence, it also acts transitively on the set of unordered triples of $V$. If $q$ is odd, then both $PSL(2,q)$ and $P\Sigma L(2,q)$ act $2$-transitively on $V$. If $q\equiv 3 \pmod 4$, then $PSL(2,q)$ acts transitively on the set of unordered triples of $V$; this explains the fact that in this case, there is no underlying two-graph. If $q\equiv 1 \pmod 4$, then both $PSL(2,q)$ and $P\Sigma L(2,q)$ have two orbits ${T}$ and ${T}'$ in their action on unordered triples, and the two orbits are swapped by $PGL(2,q)$. This means that $(V, T)$ and $(V, T')$ are isomorphic regular two-graphs with parameters $|V|=q+1$ and regularity $a=\frac{1}{2}(q-1)$. For the remainder of this section, we assume that $q\equiv 1 \pmod 4$, $G=P\Sigma L(2,q)$, and $\{0,1,\infty\} \in T$. We use the notation $\mathcal{T}=(V, T)$ for the regular two-graph of linear type. 

Fix a triple $Z=\{x,y,\infty\}\in T$. As $|PSL(2,q)|=3|T|$, the setwise stabilizer of $Z$ in $G$ acts transitively on $Z$. Swapping $x$ and $y$ if needed, we find an element $g\in G$ with $0^g=x$, $1^g=y$, and $\infty^g=\infty$. These conditions imply that $g$ has the form $g:z\mapsto (y-x)z^{p^i}+x$ with $y-x\in S_q$. Conversely, if $y-x\in S_q$, then $g\in G$ exists and $\{x,y,\infty\}\in T$. This shows that the descendant $\mathcal{T}_\infty$ is the Paley graph $Paley(q)$ plus the isolated vertex $\infty$. 

\subsection{Parameters of linear type SRGs} \label{sec:folklore-construction}
Let us assume that $\Gamma$ is a strongly regular graph in the switching class of $\mathcal{T}$. Result \ref{res:rels-twograph} implies that $\Gamma$ has parameters
\begin{align} \label{eq:linear-params}
\left(q+1,\frac{q\pm \sqrt{q}}{2},\frac{(\sqrt{q}\pm1)^2}{4}-1,\frac{(\sqrt{q}\pm1)^2}{4}\right).
\end{align}
In particular, $q$ is a square. Write $r=\sqrt q$ and view $\mathbb F_q$ as the affine plane $AG(2,r)$. Since $\mathbb F_r^\ast$ consists of squares in $\mathbb F_q^\ast$, a direction is either entirely square or entirely nonsquare on its nonzero vectors. Exactly $(r+1)/2$ of the directions are square. The lines of a square direction are cliques in the Paley graph, and each vertex has $(r-1)/2$ neighbours on every other line of that direction. The lines of a nonsquare direction are independent sets, and the corresponding number of neighbours is $(r+1)/2$.

For the plus sign, take $Y$ to be the union of $t=(r+1)/2$ parallel lines of a square direction. For the minus sign, take $t=(r-1)/2$ parallel lines of a nonsquare direction. Switch the Paley descendant $\mathcal T_\infty$ with respect to $Y$. The vertex $\infty$ acquires degree $tr$, and a finite vertex outside $Y$ has degree $(q-1)/2\pm t=tr$. For a vertex in $Y$, the degree is
\[
\begin{cases}
(t+1)(r-1)/2+(r-t)(r+1)/2+1 & \text{for a square direction},\\
(t-1)(r+1)/2+(r-t)(r-1)/2+1 & \text{for a nonsquare direction}.
\end{cases}
\]
Both expressions equal $tr$ for the respective choice of $t$. The resulting graph $P(q,Y)$ is therefore strongly regular by Result~\ref{res:rels-twograph}(ii), with parameters \eqref{eq:linear-params}.

The translations along the chosen direction fix $\infty$ and preserve $P(q,Y)$, so its automorphism group contains a subgroup of order $r=p^e$. For $q>9$, Lemma~\ref{lm:linear-maxgr} implies that a transitive automorphism group would be contained in the torus normalizer of order $(q+1)2e$. Indeed, the automorphism group avoids $PSL(2,q)$ because $PSL(2,q)$ is $2$-transitive. As $p^e$ does not divide $2e$, this is impossible.

\subsection{Transitive subgroups of \texorpdfstring{$P\Sigma L(2,q)$}{PSigmaL(2,q)}} \label{sec:psigmal-trsubgrs}
Even though the subgroup structure of $PSL(2,q)$ has been known since the beginning of the 20th century, the existence of transitive subgroups of $P\Sigma L(2,q)$ is somewhat complicated. Our primary reference is Giudici \cite{Giudici2007}. 

Fix a nonsquare element $s\in\mathbb{F}_q$. Then
\begin{align*}
C_{q-1}&=\left\{\tmat{a}{0}{0}{1} \mid a \in \mathbb{F}_q^*\right\}, \\
C_{q+1}&=\left\{\tmat{a}{b}{sb}{a} \mid a,b \in \mathbb{F}_q, (a,b)\neq (0,0)\right\}
\end{align*}
are cyclic subgroups of order $q-1$ and $q+1$ of $PGL(2,q)$. $C_{q-1}$ has two fixed points $0,\infty$, and $C_{q+1}$ acts regularly on $V$. Both of these groups have a unique subgroup of index 2:
\begin{align*}
C_{\frac{q-1}{2}}&=\left\{\tmat{a}{0}{0}{1} \mid a \in S_q\right\}, \\
C_{\frac{q+1}{2}}&=\left\{\tmat{a}{b}{sb}{a} \mid a^2-sb^2 \in S_q\right\}.
\end{align*}
Their normalizers in $PSL(2,q)$ are the dihedral groups
\[D_{q-1}=C_{\frac{q-1}{2}} \rtimes \left\langle \tmat{0}{1}{-1}{0} \right\rangle, \qquad D_{q+1}=C_{\frac{q+1}{2}} \rtimes \left\langle \tmat{1}{0}{0}{-1} \right\rangle.\]

\begin{lem} \label{lm:Dq+1-normalizer}
Let $G=P\Sigma L(2,q)$, where $q=p^f\equiv1\pmod4$ and $f\geq2$. Write $C=C_{\frac{q+1}{2}}$ and let $\phi:z\mapsto z^p$. For the fixed nonsquare $s$, put
\begin{align*}
\alpha&=\phi\tmat{s^{\frac{p-1}{2}}}{0}{0}{1}
       :z\mapsto s^{\frac{p-1}{2}}z^p,\\
\beta&=\phi\tmat{0}{1}{s^{\frac{p+1}{2}}}{0}
      :z\mapsto \frac{1}{s^{\frac{p+1}{2}}z^p}.
\end{align*}
Then
\[
N_G(C)=N_G(D_{q+1})=
\begin{cases}
C\rtimes\langle\alpha\rangle &\text{if $p\equiv1\pmod4$},\\
C\rtimes\langle\beta\rangle &\text{if $p\equiv3\pmod4$}.
\end{cases}
\]
In both cases the cyclic complement has order $2f$.
\end{lem}
\begin{proof}
Identify the point with coordinate $z=x/y$ with the $\mathbb F_q$-line spanned by $x\sqrt{s}+y$ in $\mathbb F_{q^2}$. The torus $C_{q+1}$ acts by multiplication on these lines. The map $\alpha$ is induced by the $p$th-power map, and $\rho:z\mapsto1/(sz)$ is induced by multiplication by $\sqrt s$. Thus $\alpha$ normalizes $C_{q+1}$, and $\rho$ is the unique involution of $C_{q+1}$. In particular, $\alpha$ commutes with $\rho$ and $\beta=\alpha\rho$.

The determinants of the linear parts of $\alpha$ and $\beta$ are $s^{(p-1)/2}$ and $-s^{(p+1)/2}$, respectively. Since $-1$ is a square in $\mathbb F_q$, exactly the indicated generator lies in $G$. We have $\alpha^f:z\mapsto -z$, since the $q$th-power map sends $\sqrt s$ to $-\sqrt s$. Also $\beta^2=\alpha^2$. When $p\equiv3\pmod4$, $f$ is even, so $\beta^f=\alpha^f$. In either case the indicated generator has order $2f$, projects onto the field automorphism group, and its $f$th power is the involution in the displayed complement of $C$ in $D_{q+1}$.

The group $D_{q+1}$ is the normalizer of $C$ in $PSL(2,q)$, and $C$ is characteristic in $D_{q+1}$. As $PSL(2,q)\triangleleft G$, their normalizers in $G$ coincide. The kernel of the field automorphism projection on this normalizer is $D_{q+1}$, so adjoining the indicated generator gives the entire normalizer. Its cyclic subgroup intersects $C$ trivially, since any element of the intersection has trivial field part and must lie in $\langle\alpha^f\rangle$.
\end{proof}

\begin{res}[{\cite[Theorem 1.3]{Giudici2007}}] \label{res:PSigmaL}
Let $G = P\Sigma L(2, q)$ for $q=p^f$, $p$ an odd prime and $f\geq 2$. Then the maximal subgroups of $G$ which do not contain $PSL(2, q)$ are:
\begin{enumerate}[(1)]
\item the stabilizer of a point of the projective line,
\item $N_G(D_{q-1})$ for $q\neq 9$,
\item $N_G(D_{q+1})$ for $q\neq 9$,
\item $S_5$ for $p \equiv \pm 3 \pmod{10}$ and $f = 2$,
\item $N_G(PSL(2, q_0))$ with $q = q_0^r$ for some prime $r$ (2 conjugacy classes if $r = 2$).
\end{enumerate}
\end{res}

The following lemma also treats transitive subgroups that need not be maximal in $G$.

\begin{lem} \label{lm:linear-maxgr}
Let $G=P\Sigma L(2,q)$, where $q=p^f\equiv1\pmod4$ and $f\geq2$.
\begin{enumerate}[(i)]
\item If $p\equiv1\pmod4$, every transitive subgroup of $G$ contains $PSL(2,q)$.
\item If $q=9$, there are two conjugacy classes of transitive maximal subgroups of $G$ avoiding $PSL(2,q)$, both isomorphic to $S_5$.
\item If $q>9$ and $p\equiv3\pmod4$, then $H=N_G(D_{q+1})$ is a transitive maximal subgroup of $G$. Every transitive subgroup avoiding $PSL(2,q)$ is conjugate into $H$, and every transitive subgroup of $H$ contains $C_{\frac{q+1}{2}}$.
\end{enumerate}
\end{lem}
\begin{proof}
For $q=9$, Result~\ref{res:PSigmaL} and \cite[Proposition~3.1]{Giudici2007} leave only the two classes of $S_5$. One acts on the Petersen graph, hence transitively, and the classes fuse in $P\Gamma L(2,9)$. The point stabilizer and the subfield normalizers in Result~\ref{res:PSigmaL} preserve proper projective sublines or a single point.

For $q>9$, let $A$ be transitive and put $L=PSL(2,q)$. Suppose that $A$ does not contain $L$. In $G_0=LA$, choose a maximal subgroup $M$ containing $A$. Since $M$ projects onto $G_0/L$, it cannot contain $L$. Giudici's list of novelties contains no such $G_0$ between $L$ and $P\Sigma L(2,q)$ for these values of $q$ \cite[Theorem~1.1 and Table~1]{Giudici2007}. Thus $M\cap L$ is maximal in $L$ and $M\leq N_G(M\cap L)$.

By Dickson's list \cite[Theorem~2.2]{Giudici2007}, the possibilities for $M\cap L$ are a point stabilizer, a split or nonsplit dihedral normalizer, a subfield subgroup, or $A_5$ when $q=p^2$ and $p\equiv\pm3\pmod{10}$. The point stabilizer and the split torus normalizer fix a point or a pair of points, respectively. A subfield subgroup determines its proper projective subline as the set of fixed points of its elements of order $p$, so its normalizer preserves that subline. In the $A_5$ case, the normalizer in $G$ has order dividing $120$, whereas $q+1$ does not divide $120$ for an odd square $q>9$.

It follows that $A$ lies in a conjugate of $N_G(D_{q+1})$. Write $C=C_{\frac{q+1}{2}}$. Its orbits $P_0$ and $P_\infty$ are the orbits of $0$ and $\infty$ and have equal size, since $C_{q+1}$ is regular and its involution $\rho$ interchanges $0$ and $\infty$. The map $\alpha$ fixes $0$ and preserves $P_0$, whereas $\beta$ interchanges the two orbits. Lemma~\ref{lm:Dq+1-normalizer} now gives (i) and the transitivity in (iii). Maximality follows from Result~\ref{res:PSigmaL}.

For the final assertion of (iii), let $A\leq H$ be transitive, let $r=(q+1)/2$ and put $C_0=A\cap C$. Identify $V$ with $\mathbb Z/(2r)\mathbb Z$ using a generator of $C_{q+1}$ and the base point $0$. In these coordinates, $C$ acts by even translations, $\alpha$ acts by multiplication by $p$, and $\beta$ acts by $x\mapsto px+r$. The image of $A$ in $H/C$ is generated by $\beta^j C$ for some odd $j$, since $A$ interchanges the two $C$-orbits. If $C_0<C$, choose a prime $\ell$ dividing $[C:C_0]$. Reduction modulo $\ell$ gives a transitive action of $A$ on $\mathbb Z/\ell\mathbb Z$, with $C_0$ acting trivially. Its image is therefore generated by one affine map $x\mapsto p^j x+b$ for some $b$. The order of $p$ modulo $\ell$ is even, since $p^f\equiv-1\pmod\ell$ and $\ell$ is odd. Hence $p^j\neq1\pmod\ell$, and the displayed affine map has a fixed point, a contradiction. Thus $C_0=C$.
\end{proof}

\subsection{Linear type SRGs with a transitive automorphism group} \label{sec:vertex-transitive-linear-srg}

\begin{thm} \label{thm:linear-type-srg}
Let $q=p^{2e}>9$, where $p\equiv3\pmod4$ is prime, and let $\mathcal{T}$ be the linear two-graph with full automorphism group $G=P\Sigma L(2,q)$. The switching class of $\mathcal{T}$ contains exactly two isomorphism types of vertex-transitive strongly regular graphs. They have the two parameter sets in \eqref{eq:linear-params} and are isomorphic to complements of one another. Their full automorphism groups are conjugate in $G$ to $C_{\frac{q+1}{2}}\rtimes C_{4e}$.
\end{thm}
\begin{proof}
Let $C=C_{\frac{q+1}{2}}$, $v=\infty$ and $s=\Delta_v\tau$. For a generator $g$ of $C$, a function fixed in the square bracket action satisfies $f(x^{g^{-1}})=f(x)+f(v^{g^{-1}})$ for all $x$. Iterating this identity $|C|$ times shows that $f(v^{g^{-1}})=0$, since $|C|$ is odd. Consequently the fixed subspace consists of the functions that are constant on each of $P_0$ and $P_\infty$ and zero on $P_\infty$. Its unique nonzero element $f_0$ is the indicator of $P_0$ and has $sgn(f_0)=1$.

Every $C$-orbit in the curly bracket action has odd size, so the sum of its elements is an affine combination fixed by $C$. The difference of any two affine fixed points is fixed by the square bracket action. Thus the affine fixed points are exactly $f_1$ and $f_2=f_1+f_0$, for either choice of an affine fixed point $f_1$. The normalizer $H=N_G(C)$ permutes these two functions. Their signs differ, and the sign is invariant because $|V|=q+1$ and the regularity $(q-1)/2$ are both even. Hence $H$ fixes each function.

Proposition~\ref{prop:bracket-actions}(iii) gives the two $H$-invariant graphs $\Gamma_1,\Gamma_2$ with edge cochains $\partial f_i+s$. They are regular by transitivity of $H$ and strongly regular by Result~\ref{res:rels-twograph}(ii). Their automorphism groups equal $H$, since $H$ is maximal in $G$ and neither graph can admit the $2$-transitive group $G$.

The involution $\rho$ from Lemma~\ref{lm:Dq+1-normalizer} centralizes $H$ and lies in $PGL(2,q)\setminus PSL(2,q)$. It therefore sends the complementary two-graph $\overline{\mathcal T}$ to $\mathcal T$. The relabelled complement $(\overline{\Gamma_1})^\rho$ is $H$-invariant and lies in the switching class of $\mathcal T$. Its valency differs from that of $\Gamma_1$ by \eqref{eq:linear-params}, so it equals $\Gamma_2$. This also shows that both parameter sets occur.

Finally, any transitive automorphism group of a graph in this switching class avoids $PSL(2,q)$, which is $2$-transitive. By Lemma~\ref{lm:linear-maxgr}(iii) it contains a $G$-conjugate of $C$. After relabelling within $G$, the graph is $C$-invariant and therefore equals $\Gamma_1$ or $\Gamma_2$.
\end{proof}

\section{Nonexistence results} \label{sec:nonexistence}

Recall that our aim is to classify all triples $(G,H,\Gamma)$, where $G$ is a finite 2-transitive permutation group, $H$ is a maximal transitive subgroup of $G$, $\Gamma$ is a primitive strongly regular graph, $H=\Aut(\Gamma)$, and $G$ is the automorphism group of the associated two-graph of $\Gamma$. (To be more precise, we also require that $H$ does not contain the maximal perfect subgroup of $G$.) Taylor's theorem \cite{Taylor1992} restricts the possibilities of $G$ to the groups $ASp(2m,2)$, $Sp(2m,2)$, $P\Gamma U(3,q)$, $Ree(q)$, $P\Sigma L(2,q)$, $HS$, and $Co_3$. In this section, we prove the following nonexistence results:
\begin{thm} \label{thm:nonexistence}
Let the triple $(G,H,\Gamma)$ be as in Theorem \ref{thm:main}. Then the following cases are not possible. 
\begin{enumerate}[(i)]
\item $q\geq5$ is odd and $PSU(3,q)\triangleleft G$.
\item $q=3^{2e+1}$, $e\geq1$, and $Ree(q)\triangleleft G$.
\item $Co_3\triangleleft G$.
\item $G=PSp(2m,2)$, $ab=m$, $b>1$ odd, $H=Sp(2a,2^b) \rtimes C_b$.
\item $G=PSp(2m,2)$, $H=P_m$ is the stabilizer of a maximally totally isotropic subspace.
\item $m$ even, $G=PSp(2m,2)$, $H=(Sp(m,2) \times Sp(m,2)) \rtimes C_2$. 
\end{enumerate}
\end{thm}

Table~\ref{tab:nonexistence-summary} summarizes the exclusions and the candidates that give the graphs in Table~\ref{tab:main}.

\begin{table}[ht]
\centering
\caption{Elimination of candidate triples $(G,H,\Gamma)$ versus Table \ref{tab:main}}
\label{tab:nonexistence-summary}
\small
\begin{tabular}{p{18mm}p{45mm}p{25mm}p{28mm}}
Type & $G$, $H$ & Outcome & Reference \\ \hline\hline
Affine polar & $ASp(2m,2)$, \par$H=\mathbb{F}_2^{2m}\rtimes O^\pm(2m,2)$ & survives & Table \ref{tab:main}, row 1 \\ \hline
Symplectic & $Sp(2m,2)$, $H=O^\pm(2m,2)$ & survives & Table \ref{tab:main}, row 2 \\
& $Sp(2m,2)$, $H=Sp(m,4)\rtimes C_2$, $m$ even & survives & Table \ref{tab:main}, row 2 \\
& $Sp(2m,2)$, $H=Sp(2a,2^b)\rtimes C_b$, $b>1$ odd divisor of $m$ & excluded & Thm.~\ref{thm:nonexistence}(iv) \\
& $Sp(2m,2)$, $H=P_m$ & excluded & Thm.~\ref{thm:nonexistence}(v) \\
& $Sp(2m,2)$, $H=(Sp(m,2)\times Sp(m,2))\rtimes C_2$, $m$ even & excluded & Thm.~\ref{thm:nonexistence}(vi) \\
& $Sp(6,2)$, $H=G_2(2)$ & survives & Table \ref{tab:main}, $U_3(3)$-graph \\
& $Sp(8,2)$, $H=S_{10}$ & survives & Table \ref{tab:main}, $J(10,3,1)$ \\
& $Sp(8,2)$, $H=PSL(2,17)$ & survives & Table \ref{tab:main}, graph on 136 vertices \\ \hline
Linear & $P\Sigma L(2,q)$, $H=C_{\frac{q+1}{2}}\rtimes C_{4e}$ (or $S_5$ for $q=9$) & survives & Section~\ref{sec:linear} \\ \hline
Unitary & $P\Gamma U(3,q)$ & excluded ($q\geq5$) & Thm.~\ref{thm:nonexistence}(i) \\ \hline
Ree & $Ree(q)\rtimes \Aut(\mathbb{F}_q)$ & excluded & Thm.~\ref{thm:nonexistence}(ii) \\ \hline
Sporadic & $HS$, $H=M_{22}$ & survives & Table \ref{tab:main}, $M_{22}$-graph \\
& $Co_3$ & excluded & Thm.~\ref{thm:nonexistence}(iii) \\ \hline\hline
\end{tabular}
\end{table}

Our primary tool is the seminal paper \cite {Liebeck1990} by Liebeck, Praeger, and Saxl, in which they list all triples $(G, A, B)$, where $G$ is an almost simple group with socle $L$, and $A$ and $B$ are maximal subgroups of $G$ that do not contain $L$, and $G=AB$ holds. Here, $AB=\{ab \mid a\in A, b\in B\}$, and $G=AB$ is called a factorization of $G$. If $G$ acts on $V$, and $H$ is a transitive subgroup of $G$, then we have the factorization $G=G_vH$ for any $v\in V$. Moreover, if $G$ acts primitively and $H$ is maximal, then both factors are maximal subgroups. In \cite{Liebeck1990}, the authors split the classification of factorizations $G=AB$ of almost simple groups into the following subcases:
\begin{enumerate}
\item infinite families of factorizations with $L$ classical (\cite[Tables 1, 2 and 4]{Liebeck1990});
\item exceptional factorizations with $L$ classical (\cite[Table 3]{Liebeck1990});
\item factorizations of exceptional simple groups of Lie type (\cite[Table 5]{Liebeck1990});
\item factorizations of sporadic groups (\cite[Table 6]{Liebeck1990}). 
\end{enumerate}

We focus on almost simple finite groups $G$ acting 2-transitively on $V$, and we require the subgroup $H$ not to contain the socle $L$ of $G$. In the 2-transitive groups that are relevant from our point of view, the stabilizers are as follows:
\begin{center}
\begin{tabular}{ll}
$G$  & $G_v$ \\ \hline\hline
$Sp(2m,2)$ & $O^\pm(2m,2)$ \\
$P\Gamma U(3,q)$ & parabolic subgroup $P_1$ of order $q^3(q^2-1)|\Aut(\mathbb{F}_{q^2})|$ \\
$Ree(q)$ & parabolic subgroup $P$ of order $q^3(q-1)|\Aut(\mathbb{F}_{q})|$ \\  
$P\Sigma L(2,q)$ & parabolic subgroup $P_1$ of order $\frac{1}{2}q(q-1)|\Aut(\mathbb{F}_q)|$ \\
$HS$ & $PSU(3,5) \rtimes C_2$ \\
$Co_3$ & $McL \rtimes C_2$. \\ \hline\hline
\end{tabular}
\end{center}

The following two lemmas summarize the relevant information on transitive subgroups of these 2-transitive groups. As we can see, the richest subgroup structure belongs to the symplectic group $PSp(2m,2)$. 

\begin{lem} \label{lm:sporadic-maxgr}
Let $L$ be a finite simple group. Let $G$ be a group such that $L\leq G\leq \Aut(L)$. Assume that $G$ acts 2-transitively on the set $X$. 
\begin{enumerate}[(i)]
\item If $G=HS$ is the Higman--Sims sporadic simple group and $|X|=176$, then, up to conjugacy, the unique transitive maximal subgroup of $G$ is $M_{22}$.
\item If $G=P\Gamma U(3,5)=PSU(3,5)\rtimes S_3$ acts on its $126$ isotropic projective points, no maximal subgroup of $G$ avoiding $L=PSU(3,5)$ is transitive.
\item In the following cases, no maximal subgroup of $G$ avoiding $L$ acts transitively on $X$:
\begin{enumerate}[(a)]
\item $q>5$ is an odd prime power, $L=PSU(3,q)$, and $|X|=q^3+1$.
\item $q=3^{2e+1}$, $e\geq1$, $L$ is the Ree group $Ree(q)$, and $|X|=q^3+1$.
\item $L$ is Conway's sporadic simple group $Co_3$, and $|X|=276$. 
\end{enumerate}
\end{enumerate}
\end{lem}
\begin{proof}
Parts (i) and (iii)(c) follow from Table~6 of \cite{Liebeck1990}. The unitary exclusion for $q>5$ follows from its Tables~1 and~3.

For (ii), the maximal-subgroup list for $PSU(3,5)\rtimes S_3$ \cite{AtlasV3} has only one candidate avoiding the socle with order divisible by $126$. It is the subgroup of order $126$ and structure $(7{:}3\times3){:}2$, the normalizer of a Sylow $7$-subgroup. To see that its action is intransitive, model the unitary space as $\mathbb F_{5^6}$ over $\mathbb F_{25}$ with Hermitian form $h(x,y)=\Tr_{\mathbb F_{5^6}/\mathbb F_{25}}(xy^{125})$. Multiplication by an element of order $7$ generates a Sylow $7$-subgroup, normalized by the semilinear involution $x\mapsto x^{125}$. This involution fixes $\mathbb F_{125}$ pointwise. The restriction $h(x,x)=\Tr_{\mathbb F_{125}/\mathbb F_5}(x^2)$ on $\mathbb F_{125}$ is a ternary quadratic form over $\mathbb F_5$ and has a nonzero zero. Hence the involution fixes an isotropic point. A transitive action of a group of order $126$ on $126$ points would be regular, so this subgroup is intransitive.

For (iii)(b), write $q=3^f$ with $f=2e+1\geq3$. If a transitive subgroup $H$ avoids $L$, its intersection with $L$ lies in some maximal subgroup $M$ of $L$, and $q^3+1$ divides $f|M|$ since $|G/L|$ divides $f$. By the maximal-subgroup list in \cite{Kleidman1988}, the possible orders of $M$ divide
\[
q^3(q-1),\qquad q(q^2-1),\qquad 6(q+1),\qquad 6(q+1\pm\sqrt{3q}),
\]
or $|Ree(q_0)|$ for a proper subfield $\mathbb F_{q_0}$. The factor $q^2-q+1$ of $q^3+1$ is coprime to the first three orders and exceeds $f$. Moreover,
\[
q^2-q+1=(q+1+\sqrt{3q})(q+1-\sqrt{3q}),
\]
where the two factors are coprime, coprime to $6$, and both exceed $f$. This excludes the two torus normalizers. Finally $q_0\leq q^{1/3}$ and $|Ree(q_0)|<q_0^7$, giving $f|M|<f q^{7/3}<q^3+1$. Thus no candidate can be transitive, including when $q=27$.
\end{proof}

\begin{lem} \label{lm:symplectic-maxgr}
Let $G=Sp(2m,2)$ act 2-transitively on $V^\varepsilon$ of size $2^{2m-1}+\varepsilon 2^{m-1}$, where $\varepsilon =\pm 1$. Let $H$ be a transitive maximal subgroup of $G$. Then one of the following occurs:
\begin{enumerate}[(i)]
\item $H=O^{-\varepsilon}(2m,2)$;
\item $ab=m$, $b$ prime, and $H=Sp(2a,2^b)\rtimes C_b$;
\item $\varepsilon=-1$, and $H=P_m$ is the stabilizer of a maximal totally isotropic subspace;
\item $\varepsilon=-1$, $m$ is even, and $H=(Sp(m,2) \times Sp(m,2))\rtimes C_2$ is the stabilizer of a decomposition $\mathbb{F}_2^{2m} = U \oplus U^\perp$, where $U$ is a nondegenerate subspace of dimension $m$;
\item $m=3$ and $H=G_2(2)\cong PSU(3,3)\rtimes C_2$;
\item $\varepsilon=-1$, $m=4$, and $H=S_{10}$;
\item $\varepsilon=+1$, $m=4$, and $H=PSL(2,17)$.
\end{enumerate}
\end{lem}
\begin{proof}
(i)-(iv) are infinite classes given in \cite[Table 1]{Liebeck1990}. \cite[Table 2]{Liebeck1990} tells us that for $q=2^e$, $A=G_2(q)$, $B=O^\pm(6,q)$ is an infinite class of maximal factorizations $G=AB$ of $G=PSp(6,q)$. This gives (v) with $q=2$ for both $\varepsilon=\pm 1$. (For $\varepsilon=-1$, $|V^-|=28$, and $H=G_2(2)\cong PSU(3,3)\rtimes C_2$ is doubly transitive.) (vi) and (vii) are exceptional factorizations given in \cite[Table 4]{Liebeck1990}. 
\end{proof}

\subsection{Parameters of graphs and two-graphs of symplectic type}

The rest of this section is devoted to showing that certain classes of transitive subgroups of $PSp(2m,2)$ cannot act on strongly regular graphs admitting two-graphs of symplectic type. 

Our notation is as usual: $m\geq 3$ is an integer, $\langle \cdot,\cdot \rangle$ is a nondegenerate symplectic bilinear form on $\mathbb{F}_2^{2m}$, and $G=Sp(2m,2)$ is the symplectic group preserving $\langle \cdot,\cdot \rangle$. We denote by $V^\pm$ the set of quadratic forms of type $\pm 1$, linearizing to $\langle \cdot,\cdot \rangle$; $G$ acts doubly transitively on both $V^+$ and $V^-$. Let $\mathcal{T}^\pm=(V^\pm,T^\pm)$ denote the two-graphs of symplectic type with $G=\Aut(\mathcal{T}^\pm)$; their parameters are $v^\pm=|V^\pm|=2^{2m-1}\pm 2^{m-1}$ and $a^\pm=2^{2m-2}\pm 2^{m-1}-2$; see Table \ref{tab:regularity-properties}. By Result \ref{res:rels-twograph}, the parameters of $\mathcal{T}^\pm$ allow for two sets of parameters of strongly regular graphs in the switching class of $\mathcal{T}^\pm$:
\begin{align*}
k_1^\pm&=2^{2m-2}-1, & \lambda_1^\pm &=2^{2m-3}-2, & \mu_1^\pm&=2^{2m-3} \mp 2^{m-2}; \\
k_2^\pm&=2^{2m-2}\pm 3\cdot 2^{m-2} -1, & \lambda_2^\pm &=2^{2m-3}\pm 3\cdot 2^{m-2} -2, & \mu_2^\pm&=2^{2m-3} \pm 2^{m-1} . 
\end{align*}
The first quadruple $(v^\pm,k_1^\pm,\lambda_1^\pm,\mu_1^\pm)$ gives the parameters of $NO^\mp_{2m}(2)$. If $m$ is even, then the second quadruple $(v^\pm,k_2^\pm,\lambda_2^\pm,\mu_2^\pm)$ gives the parameters of $NO^{\pm}_{m+1}(4)$. If $m$ is odd, then, in general, it is not known if an $srg(v^\pm,k_2^\pm,\lambda_2^\pm,\mu_2^\pm)$ exists. 

In both cases, the eigenvalues are $\mp 2^{m-2}-1$ and $\pm 2^{m-1}-1$. The corresponding multiplicities are
\begin{align}
f_1^\pm &= \frac{4}{3}(2^{2m-2}-1), & g_1^\pm &= \frac{1}{3}(2^{m-1}\pm 1)(2^m \pm 1), \label{eq:multilicities-1}\\
f_2^\pm = f_1^\pm +1 &= \frac{1}{3}(2^{2m}-1), & g_2^\pm = g_1^\pm -1 &= \frac{2}{3}(2^{m-2}\pm 1)(2^m \mp 1). \label{eq:multilicities-2}
\end{align}

We will use the following observation. If two distinct graphs in the same switching class admit a common transitive automorphism group $A$, the switching set between them and its complement form an $A$-invariant partition into two nonempty parts. Indeed, the symmetric difference of the edge sets is the complete bipartite graph across that partition. In particular, the action of $A$ admits a system of two blocks and $A$ has a quotient of order two.

\subsection{Field extension subgroup \texorpdfstring{$H=Sp(2a,2^b)\rtimes C_b$}{H=Sp(2a,2b).b}}
For each positive divisor $b$ of $m$, write $H(b)=Sp(2m/b,2^b)\rtimes C_b$ for the field extension subgroup. We choose compatible embeddings by viewing the underlying space as $\mathbb F_{2^m}^2$ and using the trace to $\mathbb F_{2^b}$ of its determinant form. In particular, $L_0=Sp(2,2^m)\cong PGL(2,2^m)$ is contained in every $H(b)$. It acts transitively on $V^\pm$ with stabilizers $O^\pm(2,2^m)$, by the trace identification of quadratic forms used in Lemma~\ref{lm:ranks-and-orbit-lengths-PSp(m4)}. These stabilizers are the dihedral subgroups of orders $2(2^m\mp1)$.
\begin{prop} \label{pr:Sp-2a-2b}
Assume that $m=ab$, $b>1$ is odd, and $H=Sp(2a,2^b)\rtimes C_b$. There is no $H$-invariant strongly regular graph in the switching class of $\mathcal{T}^\pm$. 
\end{prop}

The cases when $m$ is even or odd require different approaches; the odd case is harder. Our method relies on the irreducible characters of $PGL(2,q)$, $q$ even. Jordan and Schur computed the character table of $PGL(2, q)$ in 1907, and it appears in many textbooks (see \cite{James2009}). We use the notation and terminology of \cite{Meagher2011}. As above, we denote by $\tmat{a}{b}{c}{d}$ the elements of $PGL(2,q)$. Table \ref{tab:cclasses} shows the structure of the conjugacy classes; Table \ref{tab:chartable} is the character table of $PGL(2,q)$, $q$ even. The maps $\beta:\mathbb{F}_{q^2}^*/\mathbb{F}_q^* \to \mathbb{C}$ and $\gamma:\mathbb{F}_q^*\to \mathbb{C}$ are nontrivial homomorphisms from the cyclic groups of order $q\pm 1$. Moreover, $\eta_\beta=\eta_{\beta'}$ and $\nu_\gamma=\nu_{\gamma'}$ if and only if $\beta'(x)=\beta(x^{-1})$ and $\gamma'(x)=\gamma(x^{-1})$. Therefore, the number of these characters is $\frac{q}{2}$ and $\frac{q}{2}-1$, respectively. 

\begin{table}[ht]
\caption{Conjugacy classes in $PGL(2,q)$, $q$ even}
\label{tab:cclasses}
\small
\begin{tabular}{l|l|l|l|l}
Representative of the & $\tmat{1}{0}{0}{1}$ & $u=\tmat{1}{1}{0}{1}$ & $d_x=\tmat{1}{0}{0}{x}$ & $h_r=\tmat{0}{1}{r^{q+1}}{r+r^q}$ \\
conjugacy class &&& $x\in \mathbb{F}_q\setminus\{0,1\}$ & $r\in \mathbb{F}_{q^2} \setminus \mathbb{F}_q$ \\ \hline\hline
Condition of conjugacy & & & $x=y$ or $x=y^{-1}$ & $r\mathbb{F}_{q}^*=s\mathbb{F}_{q}^*$ or \\
& & & & $r\mathbb{F}_{q}^*=s^{-1}\mathbb{F}_{q}^*$ \\
Number of such classes & 1 & 1 & $\frac{1}{2}q-1$ & $\frac{1}{2}q$ \\
Size of the conjugacy class & 1 & $q^2-1$ & $q(q+1)$ & $q(q-1)$ 
\\ \hline\hline
\end{tabular}
\end{table}

\begin{table}[ht]
\caption{Character table of $PGL(2,q)$, $q$ even}
\label{tab:chartable}
\small
\begin{tabular}{l|c|c|c|c}
& 1 & $u$ & $d_x$ & $h_r$ \\ \hline\hline
$\lambda_1$ & 1 & 1 & 1 & 1 \\
$\psi_1$ & $q$ & 0 & 1 & $-1$ \\
$\eta_\beta$ ($\beta:\mathbb{F}_{q^2}^*/\mathbb{F}_q^*\to \mathbb{C}$) & $q-1$ & -1 & 0 & $-\beta(r)-\beta(r^{-1})$ \\
$\nu_\gamma$ ($\gamma:\mathbb{F}_q^*\to \mathbb{C}$) & $q+1$ & 1 & $\gamma(x)+\gamma(x^{-1})$ & 0 \\ \hline\hline
\end{tabular}
\end{table}

\begin{table}[ht]
\caption{Values of the permutation characters}
\label{tab:permchar-values}
\begin{tabular}{l|p{18mm}p{18mm}p{18mm}p{18mm}}
 & 1 & $u$ & $d_x$ & $h_r$ \\ \hline\hline
$\pi^-$ & $\frac{1}{2}q(q-1)$ & $\frac{1}{2}q$ & 0 & 1 \\
$\pi^+$ & $\frac{1}{2}q(q+1)$ & $\frac{1}{2}q$ & 1 & 0 \\ \hline\hline
\end{tabular}
\end{table}

\begin{lem} \label{lm:pgl2q-char-decomp}
Let $U^\pm$ be the dihedral subgroup $D_{2(q\mp 1)}$ of $L=PGL(2,q)$, where $q\geq8$ is even. Let $\pi^\pm=1^L_{U^\pm}$ be the permutation characters of the action of $L$ on the right cosets of $U^\pm$. Then
\begin{align*}
\pi^- &= \lambda_1 + \sum_\gamma \nu_\gamma, & \pi^+ &= \lambda_1 + \psi_1 + \sum_\gamma \nu_\gamma,
\end{align*}
where the summation is over the nontrivial homomorphisms $\gamma:\mathbb{F}_q^*\to \mathbb{C}$ up to inversion. 
\end{lem}
\begin{proof}
Table \ref{tab:permchar-values} displays the values of the permutation characters $\pi^\pm$. We immediately have $\pi^+-\pi^-=\psi_1$. Straightforward computation shows 
\[\langle \pi^-, \lambda_1 \rangle = \langle \pi^-,\nu_\gamma \rangle = 1\]
for all nontrivial homomorphisms $\gamma:\mathbb{F}_q^*\to \mathbb{C}$. Since the number of characters $\nu_\gamma$ is $\frac{q}{2}-1$, and their degree is $q+1$, we have $\pi^- = \lambda_1 + \sum_\gamma \nu_\gamma$.
\end{proof}

\begin{prop} \label{pr:multiplicity-divisibility}
Let $q\geq8$ be a power of $2$, and let $\Gamma$ be a strongly regular graph admitting the coset action of $PGL(2,q)$ on $U^\pm$ from Lemma~\ref{lm:pgl2q-char-decomp}. Write $v=\frac{1}{2}q(q\pm1)$ and let $f,g$ be the multiplicities of the restricted eigenvalues of $\Gamma$.
\begin{enumerate}
\item If $v=\frac{1}{2}q(q-1)$, then both $f,g$ are divisible by $q+1$. 
\item If $v=\frac{1}{2}q(q+1)$, then exactly one of $f,g$ is divisible by $q+1$. 
\end{enumerate}
\end{prop}
\begin{proof}
Let $r,s$ be restricted eigenvalues of $\Gamma$, with eigenspaces $R, S$. Then $R,S$ are $PGL(2,q)$-modules, and we have the decomposition $\mathbb{C}^v=\langle \mathbf{j} \rangle \oplus R \oplus S$, where $\mathbf{j}$ is the all-one vector of dimension $v$. In particular,  $R$ and $S$ do not have the trivial module as an irreducible summand. If $v=\frac{1}{2}q(q-1)$, then both characters $\chi_R, \chi_S$ are the sum of irreducible characters of degree $q+1$ by Lemma \ref{lm:pgl2q-char-decomp}. Therefore, $f=\dim(S)$ and $g=\dim(R)$ are divisible by $q+1$. If $v=\frac{1}{2}q(q+1)$, then exactly one of $\chi_R, \chi_S$ has a $q$-dimensional irreducible summand, while the other has only irreducible summands of degree $q+1$. This proves the proposition. 
\end{proof} 

We are now able to prove the nonexistence result for $H(b)$ when $b$ is an odd integer.

\begin{proof}[Proof of Proposition \ref{pr:Sp-2a-2b}]
Let $b>1$ be an odd divisor of the integer $m$, and $\Gamma$ be a putative $H(b)$-invariant strongly regular graph in the switching class of $\mathcal{T}^\pm$. 

\textbf{Case 1:} $m$ is odd. The coset action of $L_0=Sp(2,2^m)$ satisfies the hypotheses of Proposition~\ref{pr:multiplicity-divisibility}. The possible multiplicities of the eigenvalues of $\Gamma$ are given in \eqref{eq:multilicities-1} and \eqref{eq:multilicities-2}. The numerator of $f_1^\pm=(2^m-2)(2^m+2)/3$ is congruent to $-3$ modulo $2^m+1$, so $f_1^\pm$ is not divisible by $2^m+1$. Also,
\begin{align*}
f_2^\pm = \frac{1}{3}(2^m-1)(2^m+1) \;\text{ is divisible by }\; 2^m+1 \Longleftrightarrow 3\mid 2^m-1 \Longleftrightarrow \text{$m$ is even.}
\end{align*}
This shows that $\Gamma$ cannot be in the switching class of $\mathcal{T}^-$. Assume that $\Gamma$ is in the switching class of $\mathcal{T}^+$. Then $g_2^+=\frac{2}{3}(2^{m-2}+1)(2^m-1)$ is not divisible by $2^m+1$, and
\begin{align*}
g_1^+ = \frac{1}{3}(2^{m-1}+1)(2^m+1) \;\text{ is divisible by }\; 2^m+1 \Longleftrightarrow 3\mid 2^{m-1}+1 \Longleftrightarrow \text{$m$ is even.}
\end{align*}
Hence, we have a contradiction to Proposition \ref{pr:multiplicity-divisibility}.

\textbf{Case 2:} $m$ is even. The switching class of $\mathcal{T}^\pm$ contains $\Delta=NO^{\pm}_{m+1}(4)$ with automorphism group $H(2)=Sp(m,4)\rtimes C_2$. By the compatible embeddings above, $L_0$ fixes both $\Gamma$ and $\Delta$. The dihedral point stabilizers are maximal in $L_0$ by Dickson's list \cite[Theorem~2.1]{Giudici2007}, so $L_0$ acts primitively. The observation at the end of the preceding subsection implies $\Gamma=\Delta$.

However, $Sp(2a,2^b)$ is simple and could lie in $H(2)$ only by lying in its $\mathbb F_4$-linear subgroup $Sp(m,4)$. The centralizer of $Sp(2a,2^b)$ in the $\mathbb F_2$-endomorphism algebra is the scalar field $\mathbb F_{2^b}$. Indeed, an endomorphism commuting with all symplectic transvections preserves every $\mathbb F_{2^b}$-line, since each line occurs as the image of a transvection minus the identity. Additivity applied to two independent vectors then forces the endomorphism to be scalar. If $Sp(2a,2^b)$ lay in $Sp(m,4)$, its centralizer would contain the $\mathbb F_4$ scalars, which is impossible for odd $b$. Thus $H(b)$ cannot fix $\Delta$.
\end{proof}

\subsection{Stabilizer \texorpdfstring{$H=P_m$}{H=Pm} of a maximal totally isotropic subspace}
We say that a subspace $U\leq \mathbb{F}_2^{2m}$ is totally isotropic if, for any two vectors $x,y\in U$, the bilinear form $\langle x,y \rangle=0$ holds. Then $\dim(U)\leq m$, and if $\dim(U)=m$, then $U$ is a maximal totally isotropic subspace. Maximal totally isotropic subspaces are in the same $G$-orbit. The stabilizer of a maximal totally isotropic subspace in $G$ is the $m$th parabolic subgroup $P_m$ of $G$. By Lemma \ref{lm:symplectic-maxgr}(iii), $P_m$ is a maximal subgroup acting transitively on $V^-$. 

\begin{prop} \label{pr:symplectic-maxgr}
There is no $P_m$-invariant strongly regular graph in the switching class of $\mathcal{T}^-$.
\end{prop}
\begin{proof}
The main ingredient of the proof is a subgroup $L$ of $G$ with the following properties:
\begin{enumerate}
\item $L$ is contained in a (conjugate of) $P_m$ and $O^+(2m,2)$.
\item $L$ is simple.
\item $L$ is transitive on $V^-$. 
\end{enumerate}
We first show that (1)-(3) imply the proposition. Suppose that $\Gamma$ is a $P_m$-invariant graph in the switching class of $\mathcal{T}^-$. In the same switching class, let $\Delta=NO^+_{2m}(2)$ with full automorphism group $O^+(2m,2)$. A suitable conjugate of $P_m$ may be used by relabelling $\Gamma$ inside $G$. The group $L$ fixes both graphs. By the switching-set observation above, simplicity and transitivity of $L$ imply $\Gamma=\Delta$. But $P_m$ cannot lie in $O^+(2m,2)$, since the powers of $2$ dividing their orders are $2^{m^2}$ and $2^{m(m-1)+1}$, respectively.

To construct $L$, we use the hyperbolic quadratic form 
\[Q^+(x_1,\ldots,x_{2m}) = x_1x_{m+1}+\cdots+x_mx_{2m}.\]
Put
\[L=\left\{ \begin{bmatrix}A&0\\0&A^{-t}\end{bmatrix} \mid A\in GL(m,2)\right\}.\] 
Clearly, $L\cong GL(m,2)$ is simple, $L$ preserves $Q^+$, and $L$ fixes the maximal totally isotropic subspace $x_1=\cdots=x_m=0$. To show (3), we compute the stabilizer $K$ of the elliptic quadratic form 
\[Q^-(x_1,\ldots,x_{2m}) = x_1^2+x_1x_{m+1}+x_{m+1}^2+x_2x_{m+2}+\cdots+x_mx_{2m}\]
in $L$. The element $M=\begin{bmatrix}A&0\\0&A^{-t}\end{bmatrix}$ preserves $Q^-$ if and only if both $A$ and $A^t$ preserve the hyperplane $x_1=0$ of $\mathbb{F}_2^m$. This happens if and only if $A=\begin{bmatrix}1&0\\0&B\end{bmatrix}$ with $B\in GL(m-1,2)$. As $|GL(m,2):GL(m-1,2)|=(2^m-1)2^{m-1}$, the $L$-orbit of $Q^-$ has length $|V^-|=2^{2m-1}-2^{m-1}$, proving transitivity.
\end{proof}

\subsection{Stabilizer \texorpdfstring{$H=(Sp(m,2)\times Sp(m,2))\rtimes C_2$}{H=(Sp(m,2)xSp(m,2)).2} of a decomposition}
The subspace $U\leq \mathbb{F}_2^{2m}$ is nondegenerate if the restriction of $\langle \cdot,\cdot \rangle$ to $U$ is nondegenerate. Equivalently, $U\cap U^\perp=0$. Such a subspace has even dimension. Assume that $m$ is even, and let $U$ be a nondegenerate subspace of dimension $m$. Then $\mathbb{F}_2^{2m}=U\oplus U^\perp$. Let $S_{\{U,U^\perp\}}$ denote the stabilizer of this decomposition. Then $S_{\{U,U^\perp\}}\cong (Sp(m,2)\times Sp(m,2))\rtimes C_2$. By Lemma~\ref{lm:symplectic-maxgr}(iv), it is a maximal subgroup acting transitively on $V^-$.

\begin{prop} \label{pr:no-decomp-subgr}
Assume that $m$ is even. There is no $S_{\{U,U^\perp\}}$-invariant primitive strongly regular graph in the switching class of $\mathcal{T}^-$.
\end{prop}
\begin{proof}
In this proof, we will use quadratic forms on $\mathbb{F}_2^{2m}$, $U$, and $U^\perp$. We always assume that these forms linearize to $\langle \cdot,\cdot \rangle$ or its restriction. Any quadratic form $Q$ defines the forms $Q_1=Q|_{U}$ and $Q_2=Q|_{U^\perp}$. All vectors $x$ decompose as $x=u_1+u_2$ with $u_1\in U$, $u_2\in U^\perp$, and $Q(x)=Q_1(u_1)+Q_2(u_2)$. This implies that $Q_1, Q_2$ determine $Q$ uniquely. Using the standard form of quadratic forms over finite fields, it is immediate that $Q$ is elliptic whenever $Q_1$ and $Q_2$ have different types. Counting the pairs $(Q_1, Q_2)$ of different type quadratic forms, we obtain a bijection between $V^-$ and $P_1\dot\cup P_2$, where $P_1$ is the set of pairs $(Q_1, Q_2)$ with $Q_1$ an elliptic quadratic form on $U$, $Q_2$ a hyperbolic quadratic form on $U^\perp$, and $P_2$ is the set of pairs $(Q_1, Q_2)$ with $Q_1$ a hyperbolic quadratic form on $U$, $Q_2$ an elliptic quadratic form on $U^\perp$. The group $S_{\{U, U^\perp\}}$ preserves the partition $P_1\dot\cup P_2$ and acts transitively on $P_1$ and $P_2$. Even more is true: fix a pair $(Q_1, Q_2)\in P_1$; it has a stabilizer $H=O^-(m,2)\times O^+(m,2)$ in $S_{\{U, U^\perp\}}$. As $O^-(m,2)$ is transitive on the set of hyperbolic quadric forms on $U$, and $O^+(m,2)$ is transitive on the set of elliptic quadratic forms on $U^\perp$, we obtain that $H$ acts transitively on $P_2$.

Let $\Gamma$ be an $S_{\{U,U^\perp\}}$-invariant primitive strongly regular graph in the switching class of $\mathcal{T}^-$. By the last observation, a vertex in $P_1$ is adjacent to all or none of $P_2$. The subgroup $Sp(m,2)\times Sp(m,2)$ is transitive on each part, so there are either no edges or all edges across the partition. In the first case $\Gamma$ is disconnected, and in the second case $\overline{\Gamma}$ is disconnected. Each contradicts primitivity.
\end{proof}

\subsection{Proof of the main nonexistence result}

\begin{proof}[Proof of Theorem \ref{thm:nonexistence}]
Parts (i), (ii) and (iii) follow from Lemma~\ref{lm:sporadic-maxgr}(ii) and (iii). For (iv)--(vi), replace $\Gamma$ by its complement if needed to work with the indicated symplectic two-graph $\mathcal T^\pm$. Now apply Propositions~\ref{pr:Sp-2a-2b}, \ref{pr:symplectic-maxgr}, and \ref{pr:no-decomp-subgr}, respectively.
\end{proof}

\section{Uniqueness results} \label{sec:uniqueness}

In this section, we deal with the question of the uniqueness of the strongly regular graph $\Gamma$ when its associated two-graph $\mathcal{T}$, the automorphism group $G=\Aut(\mathcal{T})$, and the transitive maximal subgroup $H\leq G$ are given. At the end of the section, we complete the proof of the main result Theorem \ref{thm:main}.

\subsection{Uniqueness of affine polar graphs}

Let $\mathcal T$ be the affine two-graph with $G=ASp(2m,2)$ for $m\geq2$. When we plug in the two-graph parameters $v=2^{2m}$, $a=2^{2m-1}-2$ from Table \ref{tab:regularity-properties} into the formulas of Result \ref{res:rels-twograph}(ii), we obtain the parameters
\begin{align*}
v&=2^{2m}, & k&=2^{2m-1}\pm 2^{m-1}- 1, \\
\lambda &=2^{2m-2}\pm 2^{m-1}-2 , & \mu&=2^{2m-2}\pm 2^{m-1}
\end{align*}
of the affine polar graphs $VO_{2m}^\pm(2)$. Recall that the vertex set of $VO_{2m}^\pm(2)$ is $\mathbb{F}_2^{2m}$, and adjacency is defined by $Q(x + y) = 0$, where $Q(x)$ is a nondegenerate quadratic form of type $\pm1$; see Section \ref{sec:coboundary}. We aim to show that these graphs are the only strongly regular graphs $\Gamma$ in the switching class of $\mathcal{T}$ such that $\Aut(\Gamma)$ is a transitive maximal subgroup of $G$. We first classify the transitive maximal subgroups of $G$.

\begin{lem} \label{lm:affine-maxgr}
Let $G=ASp(2m,2)$ act 2-transitively on $V=\mathbb{F}_2^{2m}$. Let $N$ be the regular normal subgroup of $G$, and $S=G_0\cong Sp(2m,2)$ the stabilizer of $0\in V$. If $H$ is a transitive maximal subgroup of $G$, then $N\leq H$, $H=NT$, where $T$ is a maximal subgroup of $S$. 
\end{lem}
\begin{proof}
Assume that $N\not\leq H$. As $H$ is maximal, $G=NH$. The intersection $N\cap H$ is normal in both $N$ and $H$, hence in $G$. Since $N$ is minimal normal, $N\cap H=1$ and $H\cong Sp(2m,2)$. If $m\geq3$, an index $2^{2m}$ subgroup of the simple group $Sp(2m,2)$ is ruled out by \cite[Theorem~1]{Guralnick1983}. If $m=2$, an index $16$ subgroup of $Sp(4,2)\cong S_6$ would have order $45$. Its Sylow $5$-subgroup would be normal, which is impossible since the normalizer of a subgroup generated by a $5$-cycle in $S_6$ has order $20$. Thus $N\leq H$. The quotient $H/N$ corresponds to a maximal subgroup $T$ of $S$ and $H=NT$.
\end{proof}

\begin{prop} \label{pr:affine-uniqueness}
Let $m\geq 2$, $G=ASp(2m,2)$, and $H$ a transitive maximal subgroup of $G$. Let $\mathcal{T}$ be the affine two-graph of regularity $2^{2m-1}-2$, and $\Gamma$ a strongly regular graph in its switching class with $\Aut(\Gamma)=H$. Then, for some $\phi\in\{\pm1\}$, we have $H\cong\mathbb{F}_2^{2m}\rtimes O^\phi(2m,2)$ and $\Gamma\cong VO^\phi_{2m}(2)$.
\end{prop}
\begin{proof}
Choose $\epsilon$ such that $\Gamma$ has parameters 
\[(2^{2m},2^{2m-1}+\epsilon 2^{m-1}-1,2^{2m-2}+\epsilon 2^{m-1}-2,2^{2m-2}+\epsilon 2^{m-1}).\] 
Having the same associated two-graph, we can switch $\Gamma$ to $VO^{-\epsilon}_{2m}(2)$ using a set $S$. We can assume $0\in S$ by taking the complement of $S$ if necessary. By Result~\ref{res:rels-switching1}, the valencies differ and $|S|=2^{2m-1}$.

By Lemma \ref{lm:affine-maxgr}, both $\Aut(\Gamma)$ and $\Aut(VO^{-\epsilon}_{2m}(2))$ contain the minimal normal subgroup $N$ of $G$. This means that $N$ leaves the two edge sets $E(\Gamma)$, $E(VO^{-\epsilon}_{2m}(2))$ invariant. This implies that the symmetric difference of $E(\Gamma)$ and $E(VO^{-\epsilon}_{2m}(2))$ is $N$-invariant as well. Since $S$ is a switching set, this implies that $N$ preserves the partition $V=S\dot\cup\overline{S}$. Fix arbitrary elements $w,z\in S$. The translation $x\mapsto x+w$ is an element of $N$, which leaves $S$ invariant by $0,w\in S$. Hence, $w+z\in S$, and $S$ is an $\mathbb{F}_2$-linear subspace. As $|S|=2^{2m-1}$, $S$ has codimension $1$, and there is a linear functional $\alpha:V\to \mathbb{F}_2$ such that $S=\ker(\alpha)$. 

Adjacency in $VO^{-\epsilon}_{2m}(2)$ is given by $Q(x+y)=0$, where $Q(x)$ is a nondegenerate quadratic form of type $-\epsilon$. Define the quadratic form 
\[Q^*(x)=Q(x)+\alpha(x)^2=Q(x)+\alpha(x).\]
On the one hand, $Q$ and $Q^*$ polarize to the same symplectic form; hence, $Q^*$ defines a strongly regular affine polar graph $\Delta\cong VO^{\phi}_{2m}(2)$ in the switching class of $\mathcal{T}$, $\phi=\pm 1$. On the other hand, 
\[0=Q^*(x+y)=Q(x+y)+\alpha(x)+\alpha(y)\]
if and only if either $x,y$ are adjacent in $VO^{-\epsilon}_{2m}(2)$ and lie in the same part of the partition, or they are nonadjacent in $VO^{-\epsilon}_{2m}(2)$ and lie in different parts. For distinct $x,y$, this is exactly adjacency in $\Gamma$. Hence $\Gamma\cong VO^{\phi}_{2m}(2)$ and $H\cong \mathbb{F}_2^{2m}\rtimes O^\phi(2m,2)$.
\end{proof}

\subsection{Uniqueness in two-graphs of symplectic type}\label{subsec:82}

We already know four infinite classes of strongly regular graphs associated with symplectic two-graphs: $NO^\pm_{2m}(2)$ and $NO^\pm_{2m+1}(4)$. In this subsection, we prove that, besides these classes, precisely three further graphs exist. By Section \ref{sec:nonexistence}, we have a relatively short list of transitive maximal subgroups of $PSp(2m,2)$. To find the graphs that admit these subgroups as the automorphism groups, we will determine their rank and the orbit lengths of the stabilizers. 

As in Section \ref{sec:nonexistence}, $m\geq 3$ is an integer, $\langle \cdot,\cdot \rangle$ is a nondegenerate symplectic bilinear form on $\mathbb{F}_2^{2m}$, and $G=Sp(2m,2)$ is the symplectic group preserving $\langle \cdot,\cdot \rangle$. We denote by $V^\pm$ the set of quadratic forms of type $\pm 1$, linearizing to $\langle \cdot,\cdot \rangle$; $G$ acts doubly transitively on both $V^+$ and $V^-$. 

Let $H$ be a transitive maximal subgroup of $G$. By Theorem \ref{thm:nonexistence} and Lemma \ref{lm:symplectic-maxgr}, one of the following holds:
\begin{enumerate}[(A)]
\item $H=O^{\pm}(2m,2)$.
\item $m$ is even, and $H=Sp(m,4)\rtimes C_2$.
\item $m=3$, and $H=G_2(2)\cong PSU(3,3)\rtimes C_2$.
\item $m=4$, $G$ acts on 120 points, and $H=S_{10}$.
\item $m=4$, $G$ acts on 136 points, and $H=PSL(2,17)$.
\end{enumerate}
Consider the exceptional cases (C), (D), and (E). The action of $G_2(2)$ on $28$ vertices is $2$-transitive, which leaves its primitive action on $36$ vertices. The action of $S_{10}$ on the $3$-subsets of a $10$-set is primitive. The switching-set observation in Section~\ref{sec:nonexistence} gives uniqueness for each of these two primitive actions once the graph has been identified. For $PSL(2,17)$, GAP \cite{GAP} computations give uniqueness up to isomorphism. Proposition~\ref{prop:table-two-graphs} identifies the associated two-graphs, with the complement of the $U_3(3)$-graph used for the representative $\mathcal T^+$. We now treat the infinite classes (A) and (B). 
\begin{lem} \label{lm:rank-O(2m2)}
The subgroup $A=O^{\pm}(2m,2)$ of $PSp(2m,2)$ is transitive of rank $3$ on $V^\mp$. 
\end{lem}
\begin{proof}
Let $Q$ be the quadratic form associated with $A=O^\epsilon(2m,2)$. The action of $A$ on $V^{-\epsilon}$ is equivalent to its linear action on $\{x\mid Q(x)=1\}$. Take vectors $x,y,x',y'$ with $x\neq y$, $x'\neq y'$, and $Q(x)=Q(y)=Q(x')=Q(y')=1$. By Witt's lemma, the map $x\mapsto x'$, $y\mapsto y'$ extends to an element of $A$ if and only if $Q(x+y)=Q(x'+y')$. These two values both occur because the tangent graph is nontrivial. Thus there are two off-diagonal orbitals and the rank is three.
\end{proof}

\begin{lem} \label{lm:ranks-and-orbit-lengths-PSp(m4)}
Let $m=ab$, $q=2^b$, and $\epsilon\in\{1,-1\}$.
\begin{enumerate}[(i)]
\item The field extension subgroup $B=Sp(2a,q)$ acts transitively on both $V^+$ and $V^-$.
\item For $x^\epsilon\in V^\epsilon$, the stabilizer $B_{x^\epsilon}$ is isomorphic to $O^\epsilon(2a,q)$.
\item On $V^\epsilon$, the point stabilizer has the singleton orbit, a nonzero singular orbit of size $(q^{a-1}+\epsilon)(q^a-\epsilon)$ if this number is positive, and $q/2-1$ further orbits of size $q^{a-1}(q^a-\epsilon)$.
\end{enumerate}
\end{lem}
\begin{proof}
Let $Q$ be a quadratic form of type $\epsilon$ over $\mathbb F_q$, with associated symplectic form $\langle\cdot,\cdot\rangle_q$. Taking the absolute trace gives a bijection from the $\mathbb F_q$-quadratic forms with this polar form to the $\mathbb F_2$-quadratic forms with polar form $\Tr\langle\cdot,\cdot\rangle_q$. This bijection is $B$-equivariant and preserves type \cite[Lemma~4.6]{Nagy2024}. Transitivity and the point stabilizer follow from the classification of the forms over $\mathbb F_q$.

In the linear coordinates given by
\[
Q_u(x)=Q(x)+\langle u,x\rangle_q^2,
\]
the stabilizer of $Q$ is $O^\epsilon(2a,q)$ acting on $u$. The sign of the character sum of $\Tr Q_u$ is the sign of that of $\Tr Q$ multiplied by $(-1)^{\Tr Q(u)}$. Indeed, $\Tr(z^2)=\Tr(z)$, and completing the square replaces $x$ by $x+u$ in the sum. Thus $Q_u$ belongs to the same type as $Q$ if and only if $\Tr Q(u)=0$.

By Witt's lemma, the orbits of $O^\epsilon(2a,q)$ on vectors are the singleton $\{0\}$ and the nonempty level sets of $Q$ on nonzero vectors. The number of nonzero singular vectors is
\[
(q-1)|\mathcal Q^\epsilon|=q^{2a-1}-1+\epsilon q^{a-1}(q-1)
=(q^{a-1}+\epsilon)(q^a-\epsilon),
\]
where $\mathcal Q^\epsilon$ is the projective quadric of $Q$. All nonzero levels have the same size because squaring permutes $\mathbb F_q^\ast$. Subtracting the zero level from $q^{2a}$ and dividing by $q-1$ gives $q^{a-1}(q^a-\epsilon)$. Of these nonzero levels, exactly $q/2-1$ have absolute trace zero.
\end{proof}

\begin{prop} \label{pr:symplectic-uniqueness}
Let $\Gamma$ be a strongly regular graph in the switching class of $\mathcal T^\epsilon$ on $V^\epsilon$, with $m\geq3$ and $\epsilon\in\{1,-1\}$.
\begin{enumerate}[(i)]
\item If $\Aut(\Gamma)$ is conjugate in $G$ to the orthogonal subgroup $O^{-\epsilon}(2m,2)$ from Lemma~\ref{lm:symplectic-maxgr}, then $\Gamma\cong NO^{-\epsilon}_{2m}(2)$.
\item If $m$ is even and $\Aut(\Gamma)$ is conjugate in $G$ to the field extension subgroup $Sp(m,4)\rtimes C_2$, then $\Gamma\cong NO^\epsilon_{m+1}(4)$.
\end{enumerate}
\end{prop}
\begin{proof}
In (i), the subgroup is of rank three by Lemma~\ref{lm:rank-O(2m2)}. Its nontrivial invariant graphs are the known graph $NO^{-\epsilon}_{2m}(2)$ and its complement. Exactly one has two-graph $\mathcal T^\epsilon$.

In (ii), set $a=m/2$ and $q=4$ in Lemma~\ref{lm:ranks-and-orbit-lengths-PSp(m4)}. The $Sp(m,4)$ point stabilizer has two nontrivial orbits, of sizes
\[
2^{2m-2}+\epsilon\,3\cdot2^{m-2}-1
\qquad\text{and}\qquad
2^{m-2}(2^m-\epsilon).
\]
The first size is $k_2^\epsilon$. Neither size equals the other feasible valency $k_1^\epsilon=2^{2m-2}-1$. Thus the neighbourhood is precisely the first orbit. Since $NO^\epsilon_{m+1}(4)$ realizes that orbit, this determines $\Gamma$ up to isomorphism.
\end{proof}

\subsection{Proof of the main theorem}

\begin{proof}[Proof of Theorem \ref{thm:main}]
Let $\Gamma$, $\mathcal T$, $H$, and $G$ satisfy the hypotheses of the theorem. If $\mathcal T$ were trivial, then $\Gamma$ would be switching equivalent to the empty or complete graph. It would be complete bipartite or a union of at most two cliques, contrary to primitivity. We may therefore apply Taylor's classification. We may also replace $\Gamma$ by its complement to work with the representative $\mathcal T$ of either complementary pair in that classification.

In affine polar type the assertion follows from Proposition~\ref{pr:affine-uniqueness}. In linear type, \eqref{eq:linear-params} forces $q$ to be a square, and Lemma~\ref{lm:linear-maxgr} forces the characteristic to be $3$ modulo $4$. For $q=9$ the result is the Petersen graph up to complementation. For $q>9$ it follows from Theorem~\ref{thm:linear-type-srg}. Theorem~\ref{thm:nonexistence}(i)--(iii) excludes the unitary type, the Ree type, and the action of $Co_3$. For $G=HS$, the only transitive maximal subgroup up to conjugacy is $M_{22}$. The known graph from Section~\ref{subsec:M22-graph} has the required two-graph. Since $M_{22}$ is simple and transitive, the switching-set observation in Section~\ref{sec:nonexistence} shows that it fixes only that graph in this switching class.

For symplectic type, Lemma~\ref{lm:symplectic-maxgr} gives the transitive maximal subgroups and Theorem~\ref{thm:nonexistence}(iv)--(vi) excludes three families. Proposition~\ref{pr:symplectic-uniqueness} treats the orthogonal and $\mathbb F_4$ field extension subgroups. The remaining cases are $G_2(2)$ on $36$ vertices, $S_{10}$ on $120$ vertices, and $PSL(2,17)$ on $136$ vertices. Proposition~\ref{prop:table-two-graphs} identifies their switching classes. The prescribed subgroup fixes only one graph in the indicated switching class in each case, as shown in Section~\ref{subsec:82}.
\end{proof}

\section*{Acknowledgments}
We thank P\'al Heged\"us (Budapest University of Technology and Economics, Hungary) and Peter M\"uller (University of W\"urzburg, Germany) for helpful discussions. 
The first author is supported by an NSERC Discovery Grant.
The second author is supported by the National Research, Development and Innovation Fund of the Ministry for Innovation and Technology of Hungary under grant SNN 152582. 
The research of the third author was supported by the Italian National Group for Algebraic and Geometric Structures and their Applications (GNSAGA - INdAM) and by the project SID \emph{Linear Algebra methods in Combinatorics and Industrial Engineering: Coding Theory, Graphs and their interest in Assembly Line Problems} CUP C33C25001110005 of Dipartimento di Tecnica e Gestione dei Sistemi Industriali of the Universit\`a degli Studi di Padova.

\section*{AI Declaration}
The authors used OpenAI's ChatGPT (Codex) in revising this manuscript.
It reviewed the mathematical arguments, identified gaps, and helped develop
and draft revised proofs. This included the uniqueness proof for the
linear family, the treatment of the unitary case $q=5$ and the Ree case $q=27$,
and the symplectic orbit count. It also contributed the direct switching
argument for $J(10,3,1)$, and assisted with editing the
\LaTeX{} source. The authors are responsible for the final manuscript,
including its proofs and computational claims.

\bibliographystyle{abbrv}
\bibliography{twographs2025}
\end{document}